\documentclass{article}
\usepackage[english]{babel}
\usepackage{amsthm}
\usepackage{comment}
\usepackage[utf8]{inputenc}
\usepackage{float}
\usepackage{amssymb}
\usepackage{amsmath}
\usepackage[T1]{fontenc}
\usepackage[colorlinks]{hyperref}
\usepackage{amsxtra}
\usepackage{amstext}
\usepackage{amssymb}
\usepackage{amsthm}
\usepackage{graphicx}
\usepackage{amsfonts}
\usepackage{multicol}
\usepackage{setspace}
\usepackage[usestackEOL]{stackengine}
\usepackage[noabbrev]{cleveref}
\usepackage{MnSymbol}

\theoremstyle{plain}
\newtheorem{theorem}{Theorem}[section]

\newtheorem{lemma}[theorem]{Lemma}
\newtheorem{proposition}[theorem]{Proposition}
\newtheorem{claim}[theorem]{Claim}

\theoremstyle{definition}
\newtheorem{definition}[theorem]{Definition}

\theoremstyle{remark}
\newtheorem{remark}{Remark}

\newtheorem{example}{Example}

\title{Concentration inequality around the thermal equilibrium measure of Coulomb gases}
\author{David Padilla-Garza}

\begin{document}

\maketitle
\begin{abstract}
    This article deals with Coulomb gases at an intermediate temperature regime, which are governed by a Gibb's measure in which the inverse temperature is much larger than $\frac{1}{N},$ where $N$ is the number of particles. Our main result is a concentration inequality around the thermal equilibrium measure, stating that with probability exponentially close to $1,$ the empirical measure is $\mathcal{O}\left(  \frac{1}{N^{\frac{1}{d}}}\right)$ close to the thermal equilibrium measure. We also prove that this concentration inequality is optimal in some sense. The main new tool are functional inequalities that allow us to compare the bounded Lipschitz norm of a measure to its $H^{-1}$ norm in some cases when the measure does not have compact support.
\end{abstract}
\section{Introduction}

\subsection{Introduction to Coulomb gases}

Coulomb gases are a system of particles that interact via a repulsive kernel, and are confined by an external potential. Let $X_{N}=(x_{1}, x_{2},...x_{N})$ with $x_{i} \in \mathbf{R}^{d}$ and let
\begin{equation}\label{generalformulaforcoulombgases}
    \mathcal{H}_{N}\left( X_{N}\right)=\frac{1}{2} \sum_{1\leq i\neq j \leq N}g\left( x_{i}-x_{j} \right)+N \sum_{i=1}^{N} V\left( x_{i} \right),
\end{equation}
where
\begin{equation}\label{coulombkernel}
\begin{cases}
g(x)=\frac{\overline{c}_{d}}{|x|^{d-2}} \text{ if } d \geq 3, \\
g(x)=-\overline{c}_{2}\log(|x|) \text{ if } d = 2
\end{cases}
\end{equation}
is the Coulomb kernel, i.e. $g$ satisfies
\begin{equation}
    \Delta g = \delta_{0},
\end{equation} 
where the laplacian operator $\Delta$ is defined as the divergence of the gradient. In equation \eqref{coulombkernel}, $\overline{c}_{d}$ is a constant that depends only on $d$. If $d=1,$ then $-\log(|x|)$ is not the fundamental solution of laplacian. Systems given by \eqref{generalformulaforcoulombgases} with $g(x)=-\log(|x|)$ in $d=1$ are called log gases. Consider the Gibbs measure on $\mathbf{R}^{d \times N}$
\begin{equation}\label{Gibbsmeasure}
    d\mathbf{P}_{N}=\frac{1}{Z_{{N}, \beta}} \exp\left( -  \beta \mathcal{H}_{N} \right) d X_{N},
\end{equation}
where
\begin{equation}
    Z_{N, \beta} = \int_{\mathbf{R}^{d \times N}}  \exp\left( - \beta \mathcal{H}_{N} \right) d\, X_{N}.
\end{equation}
In this notation, $X_{N}=\left( x_{1}, x_{2}, ... x_{N} \right)$, $\beta$ is the inverse temperature (which we assume to depend on $N$), and  $dX_{N}= dx_{1} dx_{2} ... dx_{N} $.

We will use the notation
\begin{equation}\label{meanfield}
    \mathcal{I}_{V}\left( \mu \right) = \frac{1}{2} \iint_{\mathbf{R}^{d} \times \mathbf{R}^{d}} g(x-y) d  \mu(x) \, d \mu(y) + \int_{\mathbf{R}^{d}} V(x) \, dx
\end{equation}
for the mean-field limit of $\mathcal{H}_{N}.$ The functional \ref{meanfield} has a unique minimizer in the space of probability measures, called the equilibrium measure and denoted by $\mu_{V}$ (see \cite{serfaty2015coulomb}). 
The empricial measure $\text{emp}_{N}$ is defined as 
\begin{equation}\label{definitionofempN}
    \text{emp}_{N}= \frac{1}{N}\sum_{i=1}^{N} \delta_{x_{i}}.
\end{equation}

Throughout the paper, we will use the notation
\begin{equation}\label{definitionofE}
    \mathcal{E}(\nu)=\iint_{\mathbf{R}^{d} \times \mathbf{R}^{d} \setminus \Delta} g(x-y) d  \nu(x) \, d \nu(y),
\end{equation}
where 
\begin{equation}
    \Delta = \{ (x,x) \in \mathbf{R}^{d} \times \mathbf{R}^{d} \}.
\end{equation}

The main purpose of this paper is to prove a rate of convergence of the empirical measure to the thermal equilibrium measure, defined as 
\begin{equation}\label{muthermal}
    \mu_{\beta}=\text{argmin}_{\mu} \left\{ \mathcal{I}_{V}\left( \mu \right)+\frac{1}{N\beta}\int_{\mathbf{R}^{d}} \mu \log \left( \mu \right) \right\},
\end{equation}
where $\mathcal{I}_{V}$ is given by \eqref{meanfield} and $\mu$ is taken on the set of probability measures. For existence and uniqueness of $\mu_{\beta}$ given by \eqref{muthermal}, see \cite{armstrong2019thermal}.

\section{Statement of main results}

We will need the following hypotheses on $\mu_{V}$:
\begin{itemize}
    \item[1.] $\mu_{V}$ has compact support, denoted by $\Sigma.$ 
    \item[2.] $\mu_{V}$ has a density with respect to Lebesgue measure which has $L^{\infty}$ regularity.
\end{itemize}
We also need the following hypotheses on $\beta:$
\begin{itemize}
    \item[1.] $N\beta \to \infty$
    \item[2.] There exists a compact set $K$ such that
    \begin{equation}\label{lastpropertyofbeta}
    \begin{cases}
            \int_{\mathbf{R}^{d} \setminus K} \exp\left( -cN\beta V(x) \right) \, dx \ll \frac{1}{N^{\frac{1}{d}}} \text{ if } d \geq 3,\\
            \int_{\mathbf{R}^{d} \setminus K} \exp\left( -cN\beta [V(x)-\log(|x|)] \right) \, dx \ll \frac{1}{N^{\frac{1}{2}}} \text{ if } d = 2,
    \end{cases}
    \end{equation}
    for all $c>0,$ where the notation $\ll$ means $o().$
     \item[3.] There exists a compact set $K$ such that
    \begin{equation}\label{lastpropertyofbeta2}
    \begin{cases}
        \mathcal{E}\left(\exp\left( -cN\beta V \right)\mathbf{1}_{\mathbf{R}^{d} \setminus K}\right)\ll \frac{1}{N^{\frac{1}{d}}},\text{ if } d \geq 3,\\
         \mathcal{E}\left(\exp\left( -cN\beta [V-\log(|x|)] \right)\mathbf{1}_{\mathbf{R}^{d} \setminus K}\right)\ll \frac{1}{N^{\frac{1}{2}}},\text{ if } d= 2.
    \end{cases}
    \end{equation}
    for all $c>0.$
\end{itemize}

Note that conditions $3$ is not a consequence of conditions $1$ and $2.$ As a counterexample take $\beta=\frac{\log(\log(N))}{N}.$ However, these are not very restrictive hypotheses. For example, they are satisfied if
\begin{equation}
    \beta=\frac{1}{N^{\alpha}} \quad \alpha \in (0,1)
\end{equation}
and there exists some compact set $K \subset \mathbf{R}^{d}$ and $c>0$ such that
\begin{equation}
    V(x) \geq c|x|^{s}
\end{equation}
for $x \notin K.$

Lastly, we need the following hypotheses on the potential $V:$
\begin{itemize}\label{conditionsonV}
    \item[1.] $V$ has $C^{2}$ regularity
    \item[2.] $V$ has growth at infinity
    \begin{equation}
        \begin{split}
            \lim_{|x| \to \infty} V(x)=\infty \ &\text{ if } d\geq 3\\
            \liminf_{|x| \to \infty} V(x)-\log(|x|)=\infty \ &\text{ if } d= 2.
        \end{split}
    \end{equation}
    \item[3.] $V$ is bounded from below. Without loss of generality, we assume $V \geq 0.$
    \item[4.] 
    \begin{equation}
        \int_{|x|\geq 1} \exp\left( -N\beta V \right) \, dx < \infty
    \end{equation}
    if $d \geq 3,$ and
    \begin{equation}
        \int_{|x|\geq 1} \exp\left( -N\beta V -\log(|x|) \right) \, dx < \infty
    \end{equation}
    if $d=2.$
\end{itemize}

We need a brief definition before stating the main theorem.
\begin{definition}
For any real valued function $f$ which is measurable and weakly differentiable, define the $W^{1,\infty}$ norm as
\begin{equation}
    \|f\|_{W^{1,\infty}} = \max \{\|f\|_{L^{\infty}}, \|\nabla f\|_{L^{\infty}}\}.
\end{equation}
Define the space $W^{1,\infty} (\mathbf{R}^{d})$ as $W^{1,\infty} (\mathbf{R}^{d})=\{f| \|f\|_{W^{1,\infty}} < \infty\}$. For a measure $\mu$ on $\mathbf{R}^{d},$ define the bounded Lipschitz norm as
\begin{equation}
    \|\mu \|_{BL}=\sup_{f \in W^{1,\infty} (\mathbf{R}^{d})} \frac{\int f \, d\mu}{\|f\|_{W^{1\infty}}}.
\end{equation}
\end{definition}

Our reason for working with this norm is that the topology it induces is equivalent to the topology of weak convergence, as we state in the following remark. For a proof, see \cite{alastuey1981classical}.
\begin{remark}
Let $\mu_{N}$ be a sequence of measures on $\mathbf{R}^{d}.$ Then $\mu_{N} \to \mu$ weakly in the sense of measures, i.e.
\begin{equation}
    \mu_{N}(A) \to \mu(A)
\end{equation}
for any measurable $A \subset \mathbf{R}^{d}$ such that $\partial A$ has measure $0$, if and only if $\|\mu_{N} - \mu\|_{BL}\to 0.$ 
\end{remark}

The main result in this paper is this theorem:
\begin{theorem}\label{rateofconvergencethermalequilibrium}
Let $d \geq 2$ and assume that $\frac{1}{N} \ll \beta,$ let
\begin{equation}
    \text{emp}_{N}=\frac{1}{N} \sum_{i=1}^{N} \delta_{x_{i}},
\end{equation}
then there exists a constant $k^{*}>0$ such that
\begin{equation}
   \lim_{N \to \infty} \mathbf{P}_{N, \beta} \left( \parallel  \text{emp}_{N}-\mu_{\beta} \parallel_{BL} \leq \frac{k^{*}}{N^{\frac{1}{d}}} \right)=1.
\end{equation}
More specifically, there exist constants $c_{1}, c_{2}, c_{3}>0$ such that for any $k>0,$ we have
\begin{equation}
    \mathbf{P}_{N, \beta} \left( \parallel  \text{emp}_{N}-\mu_{\beta} \parallel_{BL} \leq\frac{k}{N^{\frac{1}{d}}} \right) \geq 1- \exp\left(- N^{2-\frac{2}{d}}\beta\left( c_{1}(k-c_{2})_{+}^{2}-c_{3} \right) \right), 
\end{equation}
where
\begin{equation}
    x_{+}=\frac{1}{2} \left( x+|x| \right).
\end{equation}
\end{theorem}

\section{Applications and motivation}

Coulomb gases have a wide range of applications in physics and mathematics, see \cite{serfaty2017microscopic} for a further discussion. Let us remark that despite the wide attention that Coulomb gases have received, the regime $\frac{1}{N} \ll \beta \ll 1$ remains largely unexplored.

Coulomb gases have applications in Statistical Physics and Quantum Mechanics (\cite{alastuey1981classical},\cite{jancovici1993large},\cite{jancovici1995classical},\cite{sari1976nu}, \cite{dean2016noninteracting}, \cite{girvin2005introduction},\cite{stormer1999fractional},  \cite{sandier2008vortices}, \cite{penrose1972thermodynamic},\cite{jancovici1993large},\cite{lieb1972constitution},\cite{lieb2005thermodynamic}). In all cases, the interactions governed by the Gibbs measure $\mathbf{P}_{N, \beta}$ are considered difficult systems because the interactions are truly long-range, singular, and the points are not constrained to live on a lattice. As always in statistical mechanics \cite{huang1987statistical}, one would like to understand if there are phase transitions for particular values of the (inverse) temperature $\beta$. For the systems studied here, one may expect what physicists call a liquid for small $\beta$, and a crystal for large $\beta$.

Apart from its direct connection with physics, the Gibbs measure \eqref{Gibbsmeasure} is related to random matrix theory (we refer to \cite{forrester2010log} for a comprehensive treatment). Random matrix theory (RMT) is a relatively old theory, pioneered by statisticians and physicists such as Wishart, Wigner and Dyson, and originally motivated by the understanding of the spectrum of heavy atoms, see \cite{mehta2004random}. For more recent mathematical reference see \cite{anderson2010introduction},\cite{deift1999orthogonal},\cite{forrester2010log}. The main question asked by RMT is: what is the law of the spectrum of a large random matrix? As first noticed in the foundational papers of \cite{wigner1993characteristic},\cite{dyson1962statistical}, in the particular cases $d=1,2$ the Gibbs measure \eqref{Gibbsmeasure} corresponds in some particular instances to the joint law of the eigenvalues (which can be computed algebraically) of some famous random matrix ensembles: 
\begin{itemize}
    \item[$\bullet$] For $d=2$, $\beta = 2$ and $V(x) = |x|^2$, \eqref{Gibbsmeasure} is the law of the (complex) eigenvalues of an $N \times N$ matrix where the entries are chosen to be normal Gaussian i.i.d. This is called the Ginibre ensemble.
    \item[$\bullet$] For $d=1$, $\beta = 2$ and $V(x) = \frac{x^2}{2}$, \eqref{Gibbsmeasure} is the law of the (real) eigenvalues of an $N \times N$ Hermitian matrix with complex normal Gaussian i.i.d. entries. This is called the Gaussian Unitary Ensemble.
    \item[$\bullet$] For $d=1$, $\beta =1$ and $V(x) = \frac{x^2}{2}$, \eqref{Gibbsmeasure} is the law of the(real) eigenvalues of an $N \times N$ real symmetric matrix with normal Gaussian i.i.d. entries. This is called the Gaussian Orthogonal Ensemble. 
    \item[$\bullet$] For $d=1$, $\beta = 4$ and $V(x) = \frac{x^2}{2}$, \eqref{Gibbsmeasure} is the law of the eigenvalues of an $N \times N$ quaternionic symmetric matrix with normal Gaussian i.i.d. entries. This is called the Gaussian Symplectic Ensemble.
\end{itemize}

\section{Comparison with literature and discussion of the temperature regime}

This paper deals with the case $\frac{1}{N} \ll \beta.$ The cases $\beta =C  N^{\frac{2}{d}-1}$ has been studied extensively \footnote{Note that the authors may use a different definition of the Gibbs measure. Hence, $\beta = CN^{\frac{2}{d}-1}$ may correspond to $\beta$ constant.} (see for example \cite{armstrong2021local}, \cite{hardin2018large}, \cite{leble2018fluctuations}, \cite{bekerman2018clt}, \cite{leble2017large}, \cite{rougerie2016higher}, \cite{arous1997large}, \cite{claeys2021much}). The regime $\beta=\frac{c}{N}$ has also been studied in the literature (see for example \cite{garcia2019concentration}, \cite{hardy2021clt}, \cite{lambert2021poisson}).

The regime $\beta=\frac{\beta_{0}}{N}$ has been studied, for example, in  \cite{garcia2019concentration, garcia2019large}. In this case the effect of temperature is so big that particles do not converge to the equilibrium measure. More precisely,  $emp_{N}$ (defined by \eqref{definitionofempN}) converges a.s. under the Gibbs measure to $\mu_{\beta_{0}},$ defined as 
\begin{equation}
    \mu_{\beta_{0}}=\text{argmin}_{\mu} F(\mu), 
\end{equation}
where 
\begin{equation}\label{freeenergyfunctional}
    F(\mu):= \frac{1}{2} \mathcal{E}(\mu)+ \int_{\mathbf{R}^{d}} V d\mu + \frac{1}{\beta_{0}} \text{ent}[\mu],
\end{equation}
$\mathcal{E}$ was defined by \eqref{definitionofE} and the argmin is taken over probability measures. In equation \eqref{freeenergyfunctional}, $\text{ent}[\mu]$ is defined as 
\begin{equation}
    \text{ent}[\mu]=\int_{\mathbf{R}^{d}} \mu \log\mu \, dx.
\end{equation}
Moreover, $emp_{N}$ satisfies a LDP with rate function $F-\min F$(\cite{garcia2019large}).

The regime studied in the present paper stands in the middle of the two regimes studied before. This is a regime in which, unlike the $\beta = \frac{\beta_{0}}{N}$ the effect of temperature is weak enough that the particles remain confined to a compact subset, in other words, $emp_{N}$ converges weakly to $\mu_{V}$ a.s. under the Gibbs measure.

Our main result is a lower bound on the probability that the empirical measure is close to the thermal equilibrium measure. Similar results were obtained in \cite{chafai2018concentration} for the equilibrium measure. The main difference in the result is that in the current work we derive a concentration inequality around the thermal equilibrium measure, not the equilibrium measure. The main difference in the techniques is that, in our case, we must compare the bounded Lipschitz norm of a measure to its electric energy even if a measure has non-compact support. A substantial part of this paper is devoted to proving an inequality which allows this comparison.

Before beginning the proof section, we make two remarks: one is that, throughout the paper, $C$ will denote a generic constant which depends only on the input parameters, and may change from line to line. We will also make the following abuse of notation: we will not distinguish between a measure and its density. 

\section{Preliminaries}

\subsection{Approximating continuous measures by atomic measures}

It is natural to ask if it is possible to approximate the empirical measure to better accuracy that $\mathcal{O}\left(\frac{1}{N^{\frac{1}{d}}} \right).$ The next proposition shows this is not possible, at least with a family of measures that has reasonable regularity.
\begin{proposition}\label{optimalorderofmagnitude}
Let $\mu_{N}$ be a sequence of absolutely continuous probability measures, with density $d\mu_{N}.$ Assume that
\begin{equation}
     d\mu_{N}  \in L^{\infty}
\end{equation}
and that
\begin{equation}
    \text{sup}_{N} \{ \| d\mu_{N} \|_{L^{\infty}}  \} \leq M.
\end{equation}
Let 
\begin{equation}
    \nu_{N}=\frac{1}{N} \sum_{i=1}^{N} \delta_{x_{i}^{N}},
\end{equation}
Then there exists $k>0$ such that
\begin{equation}
    \parallel \nu_{N}-\mu_{N} \parallel_{BL} \geq \frac{k}{N^{\frac{1}{d}}}.
\end{equation}
\end{proposition}
\begin{proof}

Let
\begin{equation}
    \overline{X}_{N}=\bigcup_{i=1}^{N} \{x_{i}^{N}\}.
\end{equation}

For each $\lambda>0,$ define the function $ \varphi_{\lambda}: \mathbf{R}^{d} \to \mathbf{R}^{+}$ as
\begin{equation}
    \varphi_{\lambda}(x)=\left( \frac{\lambda}{N^{\frac{1}{d}}} - \text{dist}(x, \overline{X}_{N}) \right)_{+}.
\end{equation}

Note that,
for every $\lambda>0,$ the function $\varphi_{\lambda}$ satisfies
\begin{equation}
    \| \varphi_{\lambda} \|_{L^{\infty}}=\frac{\lambda}{N^{\frac{1}{d}}}, \quad   \| \nabla \varphi_{\lambda} \|_{L^{\infty}}=1.
\end{equation}
Also note that
\begin{equation}
    \text{supp}(\varphi_{\lambda}) = \bigcup_{i=1}^{N}  \overline{B} \left(x_{i}^{N}, \frac{\lambda}{N^{\frac{1}{d}}} \right).
\end{equation}

We will now show that for some value of $\lambda$ to be determined later, we have
\begin{equation}
    \left| \int_{\mathbf{R}^{d}} \varphi_{\lambda} \, d(\nu_{N}-\mu_{N}) \right| \geq \frac{k}{N^{\frac{1}{d}}},
\end{equation}
for $k$ to be determined later (independent of $\lambda$). 

In order to do this, we introduce the function
\begin{equation}
    \widetilde{\mu}_{N}=M \sum_{i=1}^{N} \mathbf{1}_{{B} \left(x_{i}^{N}, \frac{\lambda}{N^{\frac{1}{d}}} \right)}.
\end{equation}

We recall the abuse of notation of not distinguishing between a measure and its density. 

We now compute
\begin{equation}
    \begin{split}
      \int_{\mathbf{R}^{d}} \varphi_{\lambda} \, d(\nu_{N}-\widetilde{\mu}_{N})  &= \frac{\lambda}{N^{\frac{1}{d}}} -\int_{\mathbf{R}^{d}} \varphi_{\lambda} \, d\widetilde{\mu}_{N}\\
      &= \frac{\lambda}{N^{\frac{1}{d}}} - M \left( \sum_{i=1}^{N} \int_{\overline{B} \left(x_{i}^{N}, \frac{\lambda}{N^{\frac{1}{d}}} \right)} \left(\frac{\lambda}{N^{\frac{1}{d}}} - \text{dist}(x, \overline{X}_{N}) \right) \, dx \right)\\
      & \geq \frac{\lambda}{N^{\frac{1}{d}}} - M \left( \sum_{i=1}^{N} \int_{\overline{B} \left(x_{i}^{N}, \frac{\lambda}{N^{\frac{1}{d}}} \right)} \frac{\lambda}{N^{\frac{1}{d}}} \, dx \right)\\
      &=\frac{\lambda}{N^{\frac{1}{d}}} - M \frac{\lambda}{N^{\frac{1}{d}}} \left( N k_{d} \left( \frac{\lambda}{N^{\frac{1}{d}}} \right)^{d}\right)\\
      &=\frac{\lambda}{N^{\frac{1}{d}}} -  \frac{M\lambda^{d+1} k_{d}}{N^{\frac{1}{d}}}\\
      &=\frac{\lambda}{N^{\frac{1}{d}}}  \left( 1-Mk_{d}\lambda^{d} \right), 
    \end{split}
\end{equation}
where $k_{d}$ is the volume of the $d-$dimensional unit ball. By taking 
\begin{equation}\label{speciallambda}
    \lambda=\frac{1}{\left( 2Mk_{d} \right)^{\frac{1}{d}}},
\end{equation}
we get that
\begin{equation}
    \begin{split}
        \int_{\mathbf{R}^{d}} \varphi_{\lambda} \, d(\nu_{N}-\widetilde{\mu}_{N})  &\geq\frac{\lambda}{N^{\frac{1}{d}}}  \left( 1-Mk_{d}\lambda^{d} \right)\\
        &= \frac{1}{N^{\frac{1}{d}}} \left( \frac{1}{2\left( 2Mk_{d} \right)^{\frac{1}{d}}} \right).
    \end{split}
\end{equation}

We will now show that
\begin{equation}
       \left| \int_{\mathbf{R}^{d}} \varphi_{\lambda} \, d(\nu_{N}-\mu_{N}) \right| \geq  \int_{\mathbf{R}^{d}} \varphi_{\lambda} \, d(\nu_{N}-\widetilde{\mu}_{N})
\end{equation}
for
\begin{equation}
    \lambda = \frac{1}{\left( 2Mk_{d} \right)^{\frac{1}{d}}}.
\end{equation}

In order to show this, note that
since $\varphi_{\lambda}$ is positive, and $\|d \mu_{N}\|_{L^{\infty}} \leq M,$ we have that
\begin{equation}
    \begin{split}
        \int_{\mathbf{R}^{d}} \varphi_{\lambda} d \mu_{N} &\leq M \int_{\mathbf{R}^{d}} \varphi_{\lambda} dx\\
        &\leq M \sum_{i=1}^{N} \int_{\mathbf{R}^{d}}  \mathbf{1}_{{B} \left(x_{i}^{N}, \frac{\lambda}{N^{\frac{1}{d}}} \right)} \varphi_{\lambda} dx\\
        &= \int_{\mathbf{R}^{d}} \varphi_{\lambda} d \widetilde{\mu}_{N}.
    \end{split}
\end{equation}

We can now conclude by taking $\lambda$ as in \eqref{speciallambda}:
\begin{equation}
    \begin{split}
        \| \nu_{N}-\mu_{N}\|_{BL} &\geq \left| \int_{\mathbf{R}^{d}} \varphi_{\lambda} \, d(\nu_{N}-\mu_{N}) \right|\\
        &\geq \int_{\mathbf{R}^{d}} \varphi_{\lambda} \, d(\nu_{N}-\widetilde{\mu}_{N}) \\
        &= \frac{1}{N^{\frac{1}{d}}} \left( \frac{1}{2\left( 2Mk_{d} \right)^{\frac{1}{d}}} \right).
    \end{split}
\end{equation}
\end{proof}

\subsection{On the $H^{-1}$ norm}

This paper will make extensive use of the $H^{-1}$ norm. We begin with an introduction about its basic properties, and relation to the Coulomb energy. 

\begin{definition}
The $H^{-1}$ norm is defined for a measure $\mu$ on $\mathbf{R}^{d}$ as 
\begin{equation}
    \parallel \mu \parallel_{H^{-1}}= \sup_{f \in C^{\infty}_{0}} \frac{\int f \, d\mu}{\parallel \nabla f \parallel_{L^{2}}}.
\end{equation}
\end{definition}

We now introduce a quantity which will be a key element when comparing the bounded Lipschitz norm of a measure to its electric energy. 

\begin{definition}
Given an open bounded set  $\Omega \subset \mathbf{R}^{d},$ we also define the $H^{-1}$ norm restricted to $\Omega,$ which we define, for any measure $\mu$ on $\Omega$ as
\begin{equation}
     \parallel \mu \parallel_{H^{-1}(\Omega)}= \sup_{f \in H^{1}(\Omega)} \frac{\int f \, d\mu}{\parallel f \parallel_{H^{1}}}.
\end{equation}
In the last equation
\begin{equation}
    \|f\|_{H^{1}}^{2}=\|f\|_{L^{2}}^{2}+\|\nabla f\|_{L^{2}}^{2}.
\end{equation}
\end{definition}

Our reasons for working with this norm are
\begin{itemize}
    \item[1.] Unlike the $H^{-1}$ norm, we can directly compare this quantity to the $BL$ norm. This comparison is actually a very easy consequence of duality and Holder's inequality.

    \item[2.] Proposition \ref{reducingH-1tocomactsets} is essential to the proof of the main result, it is not clear how to obtain a similar statement for the $H^{-1}$ norm. 
\end{itemize}

A useful inequality relates the $H^{-1}$ norm to the bounded Lipschitz norm, which we will use in the statement of the theorem. 
 
\begin{proposition}\label{fromBLtoH-1}
Let $\mu$ be a measure with compact support $K.$

Then 
\begin{equation}
     \parallel \mu \parallel_{BL} \leq \frac{1}{\sqrt{d+1} |K|^{\frac{d}{2}}} \parallel \mu \parallel_{H^{-1}(K)}.
\end{equation}
\end{proposition}

\begin{proof}
Using Holder's inequality, we obtain
\begin{equation}
    \parallel f \parallel_{L^{2}} \leq |K|^{\frac{d}{2}}  \parallel f \parallel_{L^{\infty}} \quad \text{and} \quad \parallel Df \parallel_{L^{2}} \leq \sqrt{d} |K|^{\frac{d}{2}}  \parallel Df \parallel_{L^{\infty}}.
\end{equation}
Then we have that
\begin{equation}
    \parallel f \parallel_{H^{1}} \leq \sqrt{d+1} |K|^{\frac{d}{2}}  \parallel f \parallel_{W^{1,\infty}},
\end{equation}
hence, using the fact that $W^{1, \infty}(\mathbf{R}^{d})$ is dense in $H^{1}(\mathbf{R}^{d})$ we get by duality that
\begin{equation}
     \begin{split}
         \parallel \mu \parallel_{H^{-1}(K)} &= \sup_{f \in H^{1}} \frac{\int_{\mathbf{R}^{d}} f d\, \mu}{\parallel f \parallel_{H^{1}}}\\
         &= \sup_{f \in W^{1,\infty}} \frac{\int_{\mathbf{R}^{d}} f d\, \mu}{\parallel f \parallel_{H^{1}}}\\
         &\geq \frac{1}{\sqrt{d+1} |K|^{\frac{d}{2}}} \sup_{f \in W^{1,\infty}} \frac{\int_{\mathbf{R}^{d}} f d\, \mu}{\parallel f \parallel_{W^{1, \infty}}} \\
         &=  \frac{1}{\sqrt{d+1} |K|^{\frac{d}{2}}}  \parallel \mu \parallel_{BL}. 
     \end{split}
\end{equation}

\end{proof}

A simple but useful property relates the electrostatic energy of a measure to its $H^{-1}$ norm:
\begin{proposition}\label{H-1iselectrostaticenergy}
Let $\mu$ be a signed measure on $\mathbf{R}^{2}$ of bounded variation and such that
\begin{equation}
    \int_{\mathbf{R}^{d}} d \mu =0,
\end{equation}
or an arbitrary signed measure on $\mathbf{R}^{d}$ for $d\geq 3$  of bounded variation. Let $g$ be the Coulomb kernel, then
\begin{equation}
    \parallel \mu \parallel_{H^{-1}}^{2}=\mathcal{E}(\mu).
\end{equation}
\end{proposition}

\begin{proof}

 We will continue the abuse of notation of not distinguishing between and measure and its density. Without loss of generality, we may assume that $\mu$ has a density which lies in $ C^{\infty}_{0}(\mathbf{R}^{d})$, where $C^{\infty}_{0}(\mathbf{R}^{d})$ denotes the space of smooth functions with compact support. Let $f \in C^{\infty}_{0}(\mathbf{R}^{d})$ and let $h^{\mu}=g\ast \mu.$ Using integration by parts and Cauchy-Schwartz inequality we have that
\begin{equation}
    \begin{split}
        \int_{\mathbf{R}^{d}} f \mu \, dx &= \int_{\mathbf{R}^{d}} f \Delta h^{\mu} \, dx\\
        &=  \int_{\mathbf{R}^{d}} (\nabla f) \cdot  ( \nabla^{-1}\mu) \, dx + \lim_{R \to \infty} \int_{\partial B(0,R)} \frac{\partial h^{\mu}}{\partial \nu} f d \mathcal{H}^{d-1}\\
        &\leq \sqrt{\int_{\mathbf{R}^{d}} \left|\nabla f \right|^{2}\, dx \int_{\mathbf{R}^{d}} \left| \nabla^{-1} \mu\right|^{2}\, dx}+ \lim_{R \to \infty} \int_{\partial B(0,R)} \frac{\partial h^{\mu}}{\partial \nu} f d \mathcal{H}^{d-1},
    \end{split}
\end{equation}
where $\nabla^{-1} \mu$ is defined as
\begin{equation}
    \nabla^{-1} \mu=\nabla(\mu \ast g),
\end{equation}
and $\mathcal{H}^{d-1}$ denotes the $d-1$ dimensional Hausdorff measure. 

Note that for $R$ big enough, 
\begin{equation}
    \int_{\partial B(0,R)} \frac{\partial h^{\mu}}{\partial \nu} f d \mathcal{H}^{d-1}=0,
\end{equation}
therefore
\begin{equation}
        \int_{\mathbf{R}^{d}} f \mu \, dx
        \leq \sqrt{\int_{\mathbf{R}^{d}} \left| \nabla f \right|^{2}\, dx \int_{\mathbf{R}^{d}} \left| \nabla^{-1} \mu\right|^{2}\, dx}
\end{equation}
and
\begin{equation}
    \parallel \mu \parallel_{H^{-1}}^{2} \leq \mathcal{E}(\mu).
\end{equation}

In order to prove that
\begin{equation}\label{reverseinequality}
    \parallel \mu \parallel_{H^{-1}}^{2} \geq \mathcal{E}(\mu),
\end{equation}
we claim that there exists a sequence $f_{n} \in C_{0}^{\infty}(\mathbf{R}^{d})$ such that 
\begin{equation}
    \| \nabla h^{\mu} - \nabla f_{n}\|_{L^{2}} \to 0.
\end{equation}
We first deal with the case $d \geq 3.$ We assume that 
\begin{equation}\label{finiteH-1norm}
      \parallel \mu \parallel_{H^{-1}} < \infty
\end{equation}
since otherwise inequality \eqref{reverseinequality} is trivial. Equation \eqref{finiteH-1norm} implies that 
\begin{equation}
    \int_{\mathbf{R}^{d}} \, d \mu < \infty.
\end{equation}
Without loss of generality we assume
\begin{equation}
    \int_{\mathbf{R}^{d}} \, d \mu \in \{0,1\}.
\end{equation}

Now consider, for $R > 0$ a function $\varphi_{R} \in C^{\infty}_{0}$ such that
\begin{equation}
    \begin{split}
        \varphi_{R} = 1 \text{ in } B(0,R)\\
        \text{supp}( \varphi_{R}) \subset B(0, 2R).
    \end{split}
\end{equation}
Note that we can chose $\varphi_{R}$ such that
\begin{equation}
    \begin{split}
       | \nabla \varphi_{R} | \leq C R^{-1}
    \end{split}
\end{equation}
for some universal $C \in \mathbf{R}^{+}.$

We now define 
\begin{equation}
    f_{n} = h^{\mu} \varphi_{R_{n}},
\end{equation}
for a sequence $R_{n} \to \infty.$

Note that 
\begin{equation}
    \nabla  f_{n}  = \nabla h^{\mu} \varphi_{R_{n}} +  h^{\mu} \nabla \varphi_{R_{n}}.
\end{equation}
Since we are assuming
\begin{equation}
    \int_{\mathbf{R}^{d}} \, d \mu \in \{0,1\},
\end{equation}
and $\mu$ has compact support, we have that for $R$ big enough,
\begin{equation}
    |\nabla h^{\mu}|\leq \frac{C}{R^{d-1}}
\end{equation}
and 
\begin{equation}
    | h^{\mu}|\leq \frac{C}{R^{d-2}},
\end{equation}
where $C$ depends only on $d.$

Therefore
\begin{equation}
    | \nabla  f_{n} (x)| \leq C R^{1-d}
\end{equation}
for $x \in  B(0, 2R) \setminus  B(0, R).$ Therefore
\begin{equation}
\begin{split}
    \int_{B(0, 2R) \setminus  B(0, R)}  | \nabla  f_{n}|^{2} \, dx &\leq \int_{B(0, 2R) \setminus  B(0, R)}  C R^{2(1-d)} \, dx \\
    & \leq  C R^{2-d}.
\end{split}
\end{equation}
On the other hand,
\begin{equation}
\begin{split}
    \int_{\mathbf{R}^{d} \setminus B(0,R)} | \nabla h^{\mu}|^{2} &\leq C \int_{\mathbf{R}^{d} \setminus B(0,R)} R^{2(1-d)} \\
    &\leq C R^{2-d}.
\end{split}    
\end{equation}
Therefore
\begin{equation}
\begin{split}
    \int_{\mathbf{R}^{d}} |\nabla f_{n} - \nabla h^{\mu}|^{2} \, dx &\leq C \left( \int_{\mathbf{R}^{d} \setminus B(0,R)} |\nabla h^{\mu}|^{2}  +  \int_{B(0, 2R) \setminus  B(0, R)}  | \nabla  f_{n} |^{2}\right) \\
    &\leq C R^{2-d}.
\end{split}    
\end{equation}

Therefore
\begin{equation}
    \begin{split}
        \lim_{n \to \infty} \int_{\mathbf{R}^{d}} \mu f_{n} \, dx &=  \lim_{n \to \infty} \int_{\mathbf{R}^{d}} \nabla^{-1}\mu \cdot \nabla f_{n} \\
        &= \int_{\mathbf{R}^{d}} \left| \nabla^{-1}\mu \right|^{2},
    \end{split}
\end{equation}
and 
\begin{equation}
    \lim_{n \to \infty} \| \nabla f_{n} \|_{L^{2}} = \| \nabla h^{\mu} \|_{L^{2}},
\end{equation}
which implies 
\begin{equation}
\begin{split}
    \parallel \mu \parallel_{H^{-1}} &\geq \sup_{f \in C^{\infty}_{0}} \frac{\int f \, d\mu}{\parallel \nabla f \parallel_{L^{2}}}\\
    &= \sqrt{\int_{\mathbf{R}^{d}} \left| \nabla^{-1}\mu \right|^{2}}.
\end{split}    
\end{equation}

The proof in the case $d=2$ and 
\begin{equation}
    \int_{\mathbf{R}^{d}} \, d \mu =0
\end{equation}
is almost the same. We only have to note that for $\mu$ with compact support, and for $R$ big enough (depending on $\text{supp} \mu$),
\begin{equation}
    |h^{\mu}| \leq C R^{-1}
\end{equation}
and
\begin{equation}
     |\nabla h^{\mu}| \leq C R^{-2},
\end{equation}
for a constant $C$ which depends on ${\rm supp}(\mu) := K$. In order to see this, let $m$ be the mass of the positive part of $\mu$, and let ${\rm supp}(\mu) := K$, with diameter of $K$ equal to $r>0$. W.lo.g we may assume $m=1$. Then, for $R$ big enough (depending on $r$), we have
\begin{equation}
    \begin{split}
        |h^{\mu}| &\leq \left |\log(R-r) - \log(R+r) \right|\\
        &\leq \frac{2r}{R}.
    \end{split}
\end{equation}

Similarly,
\begin{equation}
    \begin{split}
        |\nabla h^{\mu}| &\leq \left |\frac{1}{R-r} - \frac{1}{R+r} \right|\\
        &\leq \frac{2r}{R^{2}}.
    \end{split}
\end{equation}

\end{proof}

Note that, in general, $ \parallel \mu|_{\Omega} \parallel_{H^{-1}} \neq \parallel \mu|_{\Omega} \parallel_{H^{-1}(\Omega)}$ for a measure $\mu$ defined on $\mathbf{R}^{d},$ however, the two quantities are related, as we show in the next proposition. This proposition is not used in the proof, but we include it since it might be of independent interest.   

\begin{proposition}
Let $K$ be a compact set in $\mathbf{R}^{d}$ with $d \geq 2,$ and let $\mu$ be a measure of bounded variation defined on $K$. Then 
\begin{itemize}
    \item[a)] If $d=2$ and
    \begin{equation}
        \int_{K} d \, \mu =0,
    \end{equation}
     then there exists a constant $C$ such that
    \begin{equation}
        \| \mu \|_{H^{-1}(K)} \leq C \sqrt{\mathcal{E}(\mu)}. 
    \end{equation}
    
    \item[b)] If $d \geq 3$, then there exists a constant $c$ such that
    \begin{equation}
      \| \mu \|_{H^{-1}(K)}   \leq c \sqrt{\mathcal{E}(\mu)}. 
    \end{equation}
    Additionally, if 
    \begin{equation}
        \int_{\mathbf{R}^{d}} \mu \neq 0,
    \end{equation}
    then there exists a constant $c$ such that
    \begin{equation}
    \sqrt{\mathcal{E}(\mu)}  \leq \frac{1}{c}\| \mu \|_{H^{-1}(K)}. 
    \end{equation}
\end{itemize}
\end{proposition}

\begin{proof}
We first prove part a) and the first inequality of part b).  Let $f\in H^{1}(K),$ and recall that there exists an extension operator, $\overline{f}$ i.e. there exists an operator $\overline{f} \in H^{1}(\mathbf{R}^{d})$ which satisfies 
\begin{equation}
\begin{split}
    \| \nabla \overline{f} \|_{L^{2}} &\leq C \|f\|_{H^{1}},\\
    \overline{f}|_{K} &= f,
\end{split}
\end{equation}
and $\overline{f}$ has compact support. 

Using Proposition \ref{H-1iselectrostaticenergy}, and the existence of the extension operator, we have that
\begin{equation}
    \begin{split}
        \sqrt{\mathcal{E}(\mu)} &= \sup_{f \in H^{1}_{0}(\mathbf{R}^{d})} \frac{\int f \, d\mu}{\parallel \nabla f \parallel_{L^{2}}}\\
        &\geq \sup_{f \in H^{1}(K)} \frac{\int {f} \, d\mu}{\parallel \nabla \overline{f} \parallel_{L^{2}}} \\
        &\geq c \sup_{f \in H^{1}(K)} \frac{\int {f} \, d\mu}{\parallel  {f} \parallel_{H^{1}}}\\
        &= c \| \mu \|_{H^{-1}(K)}.  
    \end{split}
\end{equation}

We now turn to the second inequality of part b), for which we assume that
\begin{equation}
        \int_{\mathbf{R}^{d}} \mu \neq 0.
    \end{equation}
We further assume that $\| \mu \|_{H^{-1}(K)} < \infty$ since otherwise the inequality is trivial. This implies 
\begin{equation}
    \left|\int_{K} d \mu \right| < \infty.
\end{equation}

Since both $\|\mu\|_{H^{-1}(K)}$ and $\sqrt{\mathcal{E}(\mu)}$ are homogeneous of degree $1,$ we may assume that 
\begin{equation}
    \int_{K} d\, \mu =1.
\end{equation}

Let
\begin{equation}
    h^{\mu} = \mu \ast g.
\end{equation}
Let $\varphi_{R}$ be as in the proof of Proposition \ref{H-1iselectrostaticenergy} and let 
\begin{equation}
    h^{\mu}_{R} = h^{\mu} \varphi_{R}.
\end{equation}
Proceeding as in Proposition \ref{H-1iselectrostaticenergy},
we have that for $R$ big enough, 
\begin{equation}
    |\nabla h^{\mu}_{R}(x)| \leq \frac{C}{R^{d-1}},
\end{equation}
for $x \in B(0, 2R) \setminus B(0,R),$ with $C$ independent of $R.$ Therefore
\begin{equation}\label{boundonnablahr}
     \int_{B(0,2R) \setminus B(0,R)} |\nabla h^{\mu}_{R}|^{2} \leq C R^{2-d}.
\end{equation}
Let $R_{*}$ be such that \eqref{boundonnablahr} holds and in addition,  
\begin{equation}\label{eqwithC}
\begin{split}
   & C R_{*}^{2-d} \leq m:= \min_{\mu \in \mathcal{M}(K)| \int \mu =1} \mathcal{E}(\mu),\\
    &K \subset B(0,R_{*}),
\end{split}    
\end{equation}
where the $C$ in equation \eqref{eqwithC} is the same as in equation \eqref{boundonnablahr}.

Then, using that $\mathcal{E}(\mu) = \int \left| \nabla h^{\mu}\right|^{2} $ for $d \geq 3$, we have that
\begin{equation}
\begin{split}
    \| \nabla h^{\mu}_{R_{*}} \|_{L^{2}}^{2} &= \int_{B(0,R_{*})} |\nabla h^{\mu}|^{2} + \int_{B(0,2R_{*}) \setminus B(0,R_{*})} |\nabla h^{\mu}_{R_{*}}|^{2} \\
    & \leq 2 \| \nabla h^{\mu} \|_{L^{2}}^{2}.
\end{split}    
\end{equation}

Now consider
\begin{equation}
    \begin{split}
        f_{*} &= h^{\mu}_{R_{*}}|_{K}\\
        &= h^{\mu}|_{K}.
    \end{split}
\end{equation}
By Poincare inequality, we get that
\begin{equation}
    \| f_{*} \|_{H^{1}} \leq C \| \nabla h^{\mu} \|_{L^{2}},
\end{equation}
with $C$ depending only on $R_{*},$ hence independent of $\mu.$

Then
\begin{equation}
    \begin{split}
        \| \mu \|_{H^{-1}(K)} &\geq \frac{\int f_{*} \, d\mu}{\parallel  f_{*} \parallel_{H^{1}}} \\
        & \geq c  \frac{\int h^{\mu} \, d\mu}{\parallel \nabla h^{\mu} \parallel_{L^{2}}} \\
        &= c \sqrt{\mathcal{E}(\mu)}.
    \end{split}
\end{equation}

\end{proof}

A notable challenge in this paper is that the $H^{-1}$ norm  is not local, as this simple example shows.
\begin{example}

Let $d \geq 3,$ then there exists $f$ such that $\|f\|_{H^{-1}} < \infty$ and a compact set $K \in \mathbf{R}^{d}$ such that 
    \begin{equation}
        \parallel f|_{K} \parallel_{H^{-1}} >  \parallel f \parallel_{H^{-1}}.
    \end{equation}
\end{example}

\begin{proof}

Let $\mu$ be a bump function, i.e. $\mu$ is smooth, positive, has integral $1$, and is supported in $B(0,1).$ Let  
\begin{equation}
    f_{\lambda}(x)=\mu(x)-\lambda \mu(x-10\mathbf{1}_{d}),
\end{equation}
where
\begin{equation}
    \mathbf{1}_{d}=(1,1,...1) \in \mathbf{R}^{d}.
\end{equation}
Let 
\begin{equation}
    \nu(x)=\mu(x-10\mathbf{1}_{d}),
\end{equation}
then
\begin{equation}
    \begin{split}
         \parallel f_{\lambda} \parallel_{H^{-1}}^{2} = &\iint g(x-y) d \mu_{x} d \mu_{y}-\\
         &\lambda\iint g(x-y) d \mu_{x} d\nu_{y}+\\
         &\lambda^{2} \iint g(x-y) d\nu_{x} d\nu_{y}.
    \end{split}
\end{equation}

On the other hand, for $K=\overline{B(0,5)}$ for example, 
\begin{equation}
         \parallel f_{\lambda}|_{K} \parallel_{H^{-1}}= \iint g(x-y) d \mu_{x} d \mu_{y}.
\end{equation}
Therefore for $\lambda$ small enough, we have that
    \begin{equation}
        \parallel f_{\lambda}|_{K} \parallel_{H^{-1}} >  \parallel f_{\lambda} \parallel_{H^{-1}}.
    \end{equation}

\end{proof}

\subsection{On the thermal equilibrium measure}

Before writing the proofs, we need a few properties of the thermal equilibrium measure $\mu_{\beta}.$ 
\begin{proposition}
The measure $\mu_{\beta}$ has support in the whole of $\mathbf{R}^{d}.$
\end{proposition}
\begin{proof}
See \cite{armstrong2019thermal}
\end{proof}

\begin{proposition}\label{muthetaboundeduniformly}
The measure $\mu_{\beta}$ is uniformly bounded in $L^{\infty}$ for all $N \beta >2$ if $\mu_{V}$ is bounded in $L^{\infty}$.
\end{proposition}

\begin{proof}
See \cite{armstrong2019thermal}. Note that $N \beta$ in our notation corresponds to $\beta$ in the notation of \cite{armstrong2019thermal}.
\end{proof}

Next, we derive a splitting formula expanding around $\mu_{\beta}:$
\begin{proposition}\label{splittingformula}
 The Hamiltonian $\mathcal{H}_{N}$ can be split into (rewritten as):
 \begin{equation}
     \begin{split}
         \mathcal{H}_{N}\left( X_{N} \right)=N^{2}\mathcal{E}_{\beta}\left( \mu_{\beta} \right)&+N\sum_{i=1}^{N} \zeta_{\beta} \left( x_{i} \right)\\
         &+\frac{N^{2}}{2}\iint_{ \mathbf{R}^{d} \times \mathbf{R}^{d} \setminus \Delta} g(x-y) d\left( \text{emp}_{N}-\mu_{\beta} \right)(x)d\left( \text{emp}_{N}-\mu_{\beta} \right)(y),
     \end{split}
 \end{equation}
 where
 \begin{equation}
     \mathcal{E}_{\beta}\left( \mu \right)=  \mathcal{I}_{V}\left( \mu \right)+\frac{1}{N\beta}\int_{\mathbf{R}^{d}} \mu \log \left( \mu \right)
 \end{equation}
 and
 \begin{equation}
     \zeta_{\beta}=-\frac{1}{N\beta} \log\left( \mu_{\beta} \right). 
 \end{equation}
\end{proposition}

\begin{proof}
See \cite{armstrong2021local}.
\end{proof}
   
In analogy with previous work in this field (\cite{armstrong2021local}, \cite{leble2018fluctuations}, \cite{bekerman2018clt}, \cite{leble2017large}), we define a next order partition function $K_{N, \beta},$ as 
\begin{equation}\label{definitionofnextorderpartitionfunction}
    K_{N, \beta}=Z_{N,\beta} \exp\left( N^{2}\beta \mathcal{E}_{\beta} \left( \mu_{\beta} \right) \right).
\end{equation}
Using \eqref{definitionofnextorderpartitionfunction}, we may rewrite the Gibbs measure as
\begin{equation}
  d \mathbf{P}_{N, \beta}(x_{1}...x_{N}) = \frac{1}{K_{N, \beta}} \exp\left( -\frac{1}{2} N^{2}\beta \mathcal{E}(\text{emp}_{N}-\mu_{\beta} ) \right) \Pi_{i=1}^{N} d \mu_{\beta}(x_{i}).
\end{equation}
We need an elementary bound on $K_{N, \beta},$ which can be easily deduced from \cite{armstrong2021local}.

\begin{proposition}\label{boundonpartitionfunction}
In dimension $d \geq 3,$ the next order partition function is greater than $1$, in other words,
\begin{equation}
    \log \left( K_{N, \beta} \right) \geq 0.
\end{equation}

In dimension $2,$ we have the bound
\begin{equation}
    \log(K_{N, \beta}) \geq +\frac{1}{4}  \beta N \log(N) - c_{V}\beta N,
\end{equation}
for $N \beta \geq 1$, where $c_{V}$ depends only on $V$. 
\end{proposition}

\begin{proof}

We start by characterizing the thermal equilibrium measure. A standard computation (see for example, \cite{armstrong2019thermal}) shows that $\mu_{\beta}$ satisfies the equation
\begin{equation}
    h^{\mu_{\beta}} + V + \frac{1}{N \beta} \log \mu_{\beta} = c,
\end{equation}
for some constant $c \in \mathbf{R}$. Multiplying by $\mu_{\beta}$, integrating and using that $\mu_{\beta}$ is a probability measure, we get that
\begin{equation}
     h^{\mu_{\beta}} + V + \frac{1}{N \beta} \log \mu_{\beta} = \frac{1}{2} \mathcal{E}(\mu_{\beta}) + \mathcal{E}_{\beta}(\mu_{\beta}). 
\end{equation}

We now use the variational characterization of the partition function (see for example \cite{rougerie2016higher}):
\begin{equation}
    - \frac{\log Z_{N, \beta}}{\beta} = \min_{\mu \in \mathcal{P}(\mathbf{R}^{d \times N})} \int \mathcal{H}_{N}(X_{N}) \mu(X_{N}) d X_{N} + \frac{1}{\beta} \int \mu \log \mu,
\end{equation}
where $\mathcal{P}(\mathbf{R}^{d \times N})$ denotes the space of probability measures on $\mathbf{R}^{d \times N}$. 

Taking $\mu = \mu_{\beta}^{\otimes N}$ as a trial function, and using the splitting formula, we have that
\begin{equation}
    \begin{split}
        - \frac{\log Z_{N, \beta}}{\beta} &\leq N^{2}  \mathcal{E}_{\beta}(\mu_{\beta}) - \frac{N}{2}\mathcal{E}(\mu_{\beta}) \\
        & \leq N^{2}  \mathcal{E}_{\beta}(\mu_{\beta}),
    \end{split}
\end{equation}
which implies that
\begin{equation}
    \log \left( K_{N, \beta} \right) \geq 0.
\end{equation}
Note that this equation is true also in dimension $d \geq 2$, but we will need a stronger bound in dimension $2$ in order to conclude. 

In dimension $2$, the statement follows from Theorem $2$ of \cite{serfaty2020gaussian}, or Proposition 2.13 of \cite{leble2018fluctuations}.

\end{proof}

Next we derive an elementary concentration inequality, which will be the foundation of the theorem. We will use the notation 
\begin{equation}
    \mathcal{P}_{N}=\left\{ \mu \in \mathcal{P}\left( \mathbf{R}^{n} \right) \ \mu = \frac{1}{N}\sum_{i=1}^{N} \delta_{x_{i}} \right\},
\end{equation}
where $\mathcal{P}\left( \mathbf{R}^{n} \right)$ is the set of probability measures on $\mathbf{R}^{d}.$ In other words, $\mathcal{P}_{N}$ is the set of probability measures that consist of $N$ equally weighted point masses. 

\begin{lemma}\label{elementaryconcineq}
Let $A$ be an open set in the space of probability measures. Then
\begin{equation}
  \frac{1}{\beta N^{2}} \log\left(  \mathbf{P}_{N,\beta} \left( \text{emp}_{N} \in A \right)\right) \leq - \frac{\log K_{N, \beta}}{\beta N^{2} } - \inf_{\mu \in A \cap \mathcal{P}_{N}} G(\mu-\mu_{\beta}, \mu-\mu_{\beta}),
\end{equation}
where
\begin{equation}
    G(\mu, \nu)=\iint_{\mathbf{R}^{d} \times \mathbf{R}^{d} \setminus \Delta } g(x-y) d\mu_{x} d\nu_{y}.
\end{equation}
\end{lemma}

\begin{proof}
Using Proposition \ref{splittingformula}, we start by writing
\begin{equation}
\begin{split}
    &\mathbf{P}_{N,\beta} \left( \text{emp}_{N} \in A \right)=\\
    &\frac{1}{K_{N,\beta}} \int_{\text{emp}_{N} \in A} \exp\left( -\beta \left[  N \sum_{i=1}^{N} \zeta_{\beta} \left( x_{i} \right) +\frac{N^{2}}{2} G\left( \text{emp}_{N}-\mu_{\beta}, \text{emp}_{N}-\mu_{\beta} \right)\right] \right) d X_{N}.
\end{split}
\end{equation}

Using the definition of $\zeta_{\beta},$  we have
\begin{equation}
\begin{split}
    \mathbf{P}_{N,\beta} \left( \text{emp}_{N} \in A \right) &\leq \frac{1}{K_{N,\beta}} \int_{\text{emp}_{N} \in A} \exp \left(- G\left( \text{emp}_{N}-\mu_{\beta}, \text{emp}_{N}-\mu_{\beta} \right) \right) \Pi_{i=i}^{N}\, d \mu_{\beta}(x_{i})\\ 
    &\leq \frac{1}{K_{N,\beta}}   \exp\left(- N^{2}\beta \inf_{\mu \in A \cap \mathcal{P}_{N}} G \left(\mu-\mu_{\beta}, \mu-\mu_{\beta} \right) \right) \int_{\text{emp}_{N} \in A}  \Pi_{i=i}^{N}\, d \mu_{\beta}(x_{i}) \\
    &\leq \frac{1}{K_{N,\beta}} \exp\left(- N^{2}\beta \inf_{\mu \in A \cap \mathcal{P}_{N}} G \left(\mu-\mu_{\beta}, \mu-\mu_{\beta} \right) \right).
\end{split}
\end{equation}

The proposition follows by taking $\log$ on both sides. 
\end{proof}

We need one more technical proposition. It will be based on the following result:
\begin{lemma}\label{decayofmutheta}
There exists a constant $C$ and a compact set $K$ (both depending only on $V$ and $d$) such that, for every $N \beta >2$, in dimension $3$ and higher
\begin{equation}
    \mu_{\beta}(x) \leq C\exp\left(-CN\beta V(x) \right),
\end{equation}
for $x \notin K$ and in dimension $2,$
\begin{equation}
    \mu_{\beta}(x) \leq C\exp\left(-C N\beta [V(x)-\log(|x|)] \right),
\end{equation}
for $x \notin K$. 
\end{lemma}

\begin{proof}
See \cite{armstrong2019thermal}.

\end{proof}

\section{Proofs}

Unless otherwise stated, if $\mu \in \mathcal{P}(\mathbf{R}^{d})$ and $\epsilon >0,$ the notation $B(\mu, \epsilon)$ denotes 
\begin{equation}
    B(\mu, \epsilon) = \{ \nu \in \mathcal{P}(\mathbf{R}^{d}) | \| \mu - \nu\|_{BL} < \epsilon \}.
\end{equation}

In this section we prove the results stated in Section 2. The strategy is to use the elementary concentration inequality in proposition \ref{elementaryconcineq} as a foundation. The challenge is to estimate 
\begin{equation}
    \inf_{\mu \in A \cap \mathcal{P}_{N}} G \left(\mu-\mu_{\beta}, \mu-\mu_{\beta} \right)
\end{equation}
when 
\begin{equation}
    A=\left( B\left(\mu_{\beta}, \frac{k}{N^{\frac{1}{d}}} \right)\right)^{C}.
\end{equation}
The way to do this will be to pass from atomic measures to absolutely continuous probability measures (this will make an additive error of order $\frac{1}{N^{\frac{2}{d}}}$ in the energy if $d \geq 3$ or an error of size $\frac{C}{N}+\frac{\log N}{2N}$ if $d=2$, plus an error of order $\frac{1}{N^{\frac{1}{d}}}$ in the distance to $\mu_{\beta}$), then from absolutely continuous probability measures to absolutely continuous probability measures with compact support, and then use Proposition \ref{fromBLtoH-1} (this will make a multiplicative error of a constant).

Next, we show that we can reduce to absolutely continuous probability measures. The next proposition is proved in the appendix (it is restated as Proposition \ref{approximationbysmoothfunctions2}).

\begin{proposition}\label{approximationbysmoothfunctions}
Let $\lambda=\frac{1}{d},$ and let $A_{N}=\mathcal{P}(\mathbf{R}^{d})\setminus B\left( \mu_{\beta}, \frac{k_{1}}{N^{\lambda}} \right),$ where $\mathcal{P}(\mathbf{R}^{d})$ is the set of probability measures on $\mathbf{R}^{d},$ with $d\geq 3.$ Let $\mathcal{S}$ denote the set of absolutely continuous probability measures, then there exists a constant $C$ such that, if $d \geq 3$ then
\begin{equation}
    \inf_{\mu \in A_{N} \cap \mathcal{P}_{N} } G \left(\mu-\mu_{\beta}, \mu-\mu_{\beta} \right) \geq \inf_{\mu \in B_{N}\cap  \mathcal{S}} G \left(\mu-\mu_{\beta}, \mu-\mu_{\beta} \right) - \frac{C}{N^{\frac{2}{d}}},
\end{equation}
where 
\begin{equation}
    B_{N}=\mathcal{P}(\mathbf{R}^{d})\setminus B\left( \mu_{\beta}, \frac{(k_{1}-k_{2})_{+}}{N^{\lambda}} \right),
\end{equation}
for some absolute constant $k_{2},$ where
\begin{equation}
    (x)_{+}=
    \begin{cases}
    x \text{ if } x \geq 0\\
    0 \text{ o.w.}
    \end{cases}
\end{equation}
The constant $C$ depends only on $\|\mu_{V}\|_{L^{\infty}}.$

If $d=2$ then
\begin{equation}
    \inf_{\mu \in A_{N} \cap \mathcal{P}_{N} } G \left(\mu-\mu_{\beta}, \mu-\mu_{\beta} \right) \geq \inf_{\mu \in B_{N}\cap  \mathcal{S}} G \left(\mu-\mu_{\beta}, \mu-\mu_{\beta} \right) - \frac{C}{N}-\frac{1}{2 N} \log(N).
\end{equation}
The constant $C$ depends only on $\|\mu_{V}\|_{L^{\infty}}.$

\end{proposition}

We will need the following proposition in order to reduce ourselves to probability measures with compact support. 
\begin{proposition}\label{localizingBLnorm}
Let $\nu_{N}$ be a sequence of probability measures such that 
\begin{equation}
    \parallel \nu_{N}-\mu_{\beta} \parallel_{BL} \geq \frac{\epsilon}{N^{\frac{1}{d}}},
\end{equation}
then there exists a compact set $K^{*}$ such that
\begin{equation}\label{reducingBLtocpmtactsets}
    \parallel \left( \nu_{N}-\mu_{\beta} \right) \mathbf{1}_{K^{*}} \parallel_{BL} \geq \frac{\epsilon}{4N^{\frac{1}{d}}}.
\end{equation}
Furthermore, \eqref{reducingBLtocpmtactsets} also holds for any compact set ${K}$ which contains $K^{*}.$ 
\end{proposition}

\begin{proof}

Let $K^{*}$ be a compact set as in lemma \ref{decayofmutheta} and such that property $2$ of $\beta$ (equation \eqref{lastpropertyofbeta}) holds. We define, for any compact set $K$ which contains $K^{*},$ the probability measure
\begin{equation}
    \mu_{\beta}^{K} = \frac{\mu_{\beta}|_{K}}{ \int_{K} \mu_{\beta} \, dx}.
\end{equation}

We claim that 
\begin{equation}
    \parallel \mu_{\beta}-\mu_{\beta}^{K} \parallel_{BL} \leq   2 \|\mu_{\beta}-\mu_{\beta}^{K^{*}}\|_{BL},
\end{equation}
for any $K$ that contains $K^{*}.$

To see this, note that
\begin{equation}
    \|\mu_{\beta}-\mu_{\beta}^{K} \|_{BL} \leq  \|\mu_{\beta}-\mu_{\beta}|_{K} \|_{BL}+ \|\mu_{\beta}|_{K}-\mu_{\beta}^{K} \|_{BL}.
\end{equation}
Then, we have that
\begin{equation}
    \|\mu_{\beta}-\mu_{\beta}|_{K} \|_{BL} = \int_{\mathbf{R}^{d} \setminus K} \mu_{\beta} dx,
\end{equation}
and therefore
\begin{equation}
    \|\mu_{\beta}-\mu_{\beta}|_{K} \|_{BL} \leq \|\mu_{\beta}-\mu_{\beta}|_{K^{*}} \|_{BL}.
\end{equation}
We also have that
\begin{equation}
    \begin{split}
        \|\mu_{\beta}^{K} - \mu_{\beta}|_{K}\|_{BL} &= \left(  \frac{1}{\int_{K} \mu_{\beta} dx} -1 \right) \|\mu_{\beta}|_{K} \|_{BL} \\
        &=\int_{\mathbf{R}^{d} \setminus K} \mu_{\beta} dx\\
        &\leq \|\mu_{\beta}-\mu_{\beta}|_{K^{*}} \|_{BL}.
    \end{split}
\end{equation}

Hence, we have that
\begin{equation}
        \parallel \mu_{\beta}-\mu_{\beta}^{K} \parallel_{BL} \leq  2 \parallel \mu_{\beta}-\mu_{\beta}|_{K^{*}} \parallel_{BL},
\end{equation}
Note that
\begin{equation}
    \parallel \mu_{\beta}-\mu_{\beta}|_{K^{*}} \parallel_{BL} \ll \frac{1}{N^{\frac{1}{d}}}
\end{equation}
by property $2$ of $\beta$ and Lemma \ref{decayofmutheta}. Hence, there exists an $N_{0}$ such that, for any compact set $K$ which contains $K^{*}$ and $N > N_{0}$ we have
\begin{equation}
    \begin{split}
        \parallel \nu_{N} - \mu_{\beta}^{K} \parallel_{BL} &\geq \parallel \nu_{N} - \mu_{\beta} \parallel_{BL}-\parallel \mu_{\beta}^{K} - \mu_{\beta} \parallel_{BL}\\
        &\geq \frac{2}{3} \frac{\epsilon}{N^{\frac{1}{d}}}.
    \end{split}
\end{equation}

We now claim that 
\begin{equation}
    \parallel \nu_{N}^{K} - \mu_{\beta}^{K} \parallel_{BL} \geq \frac{1}{3} \frac{\epsilon}{N^{\frac{1}{d}}},
\end{equation}
where
\begin{equation}
    \nu_{N}^{K}=\mathbf{1}_{K} \nu_{N}.
\end{equation}

To see this, note that
\begin{equation}
    \nu_{N} - \mu_{\beta}^{K} = \nu_{N}^{K} - \mu_{\beta}^{K} +\nu_{N}^{K^{C}},
\end{equation}
where
\begin{equation}
    \nu_{N}^{K^{C}}= \nu_{N} \mathbf{1}_{\mathbf{R}^{d} \setminus K}.
\end{equation}
Therefore by triangle inequality,
\begin{equation}
    \parallel \nu_{N} - \mu_{\beta}^{K} \parallel_{BL} \leq \parallel \nu_{N}^{K} - \mu_{\beta}^{K} \parallel_{BL} + \parallel \nu_{N}^{K^{C}} \parallel_{BL}.
\end{equation}
Therefore either 
\begin{equation}
     \parallel \nu_{N}^{K} - \mu_{\beta}^{K} \parallel_{BL} \geq \frac{1}{3} \frac{\epsilon}{N^{\frac{1}{d}}}
\end{equation}
or 
\begin{equation}
     \parallel \nu_{N}^{K^{C}} \parallel_{BL} \geq \frac{1}{3} \frac{\epsilon}{N^{\frac{1}{d}}}.
\end{equation}
We proceed by contradiction and assume that
\begin{equation}
     \parallel \nu_{N}^{K} - \mu_{\beta}^{K} \parallel_{BL} < \frac{1}{3} \frac{\epsilon}{N^{\frac{1}{d}}}.
\end{equation}
Then 
\begin{equation}
     \parallel \nu_{N}^{K^{C}} \parallel_{BL} > \frac{1}{3} \frac{\epsilon}{N^{\frac{1}{d}}}.
\end{equation}
Since $\nu_{N}^{K^{C}} $ is positive, we have
\begin{equation}
    \begin{split}
     \parallel \nu_{N}^{K^{C}} \parallel_{BL} &= \int_{\mathbf{R}^{d} \setminus K} \nu_{N} \, dx\\
     &> \frac{1}{3} \frac{\epsilon}{N^{\frac{1}{d}}}.   
    \end{split}
\end{equation}
Since 
\begin{equation}
    \int_{\mathbf{R}^{d}} \nu_{N} =1,
\end{equation}
we have that 
\begin{equation}
    \int_{K} \nu_{N} < 1-\frac{1}{3} \frac{\epsilon}{N^{\frac{1}{d}}}.
\end{equation}
But this means
\begin{equation}
    \begin{split}
        \parallel \nu_{N}^{K} - \mu_{\beta}^{K} \parallel_{BL} &\geq \int_{K}  \mu_{\beta}^{K}-\nu_{N} dx \\
        &\geq \frac{1}{3} \frac{\epsilon}{N^{\frac{1}{d}}}.
    \end{split}
\end{equation}
This is a contradiction and therefore
\begin{equation}
     \parallel \nu_{N}^{K} - \mu_{\beta}^{K} \parallel_{BL} \geq \frac{1}{3} \frac{\epsilon}{N^{\frac{1}{d}}}.
\end{equation}

Proceeding as before, and using property $3$ of $\beta$ and Lemma \ref{decayofmutheta}, there exists an $N_{1}$ such that, for any compact set $K$ which contains $K^{*}$ and $N> N_{1}$ we have
\begin{equation}
\parallel \mu_{\beta}^{K} - \mu_{\beta}\mathbf{1}_{K} \parallel_{BL} \leq \frac{1}{12} \frac{\epsilon}{N^{\frac{1}{d}}}.
\end{equation}

Therefore for $N \geq \max \{N_{0}, N_{1}\}$ we have 
\begin{equation}
    \begin{split}
        \parallel (\nu_{N} - \mu_{\beta})\mathbf{1}_{K} \parallel_{BL} & \geq \parallel \nu_{N}^{K} - \mu_{\beta}^{K} \parallel_{BL}-\parallel \mu_{\beta}^{K}  - \mu_{\beta}\mathbf{1}_{K} \parallel_{BL}\\
        &\geq \frac{1}{4} \frac{\epsilon}{N^{\frac{1}{d}}}.
    \end{split}
\end{equation}
\end{proof}

We need one more result, proposition \ref{reducingH-1tocomactsets}. After we prove it, the main theorem of this section will be a corollary. Proposition \ref{reducingH-1tocomactsets} is itself based on the following lemma, which is a refinement of the extension lemma for $H^{1}$ functions and will be proved in the appendix (it is restated as Lemma \ref{extensionlemma2}).

\begin{lemma}\label{extensionlemma}
Let $\Omega \subset \mathbf{R}^{d}$ be a bounded open set with a $C^{2}$ boundary, and let $f \in H^{1}({\Omega}).$ For every $\epsilon>0$ there exists $f_{\epsilon} \in H^{1}(\mathbf{R}^{d})$ such that 
\begin{itemize}
    \item[$\bullet$] The restriction satisfies $f_{\epsilon}|_{\Omega}=f.$
    \item[$\bullet$] The support satisfies $\mbox{supp}(f_{\epsilon}) \subset \overline{\Omega}_{\epsilon},$ where
    \begin{equation}
        \Omega_{\epsilon}=\{ x \in \mathbf{R}^{d}| d(x,\Omega) < \epsilon\}.
    \end{equation}
    \item[$\bullet$] We have control of the norms:
    \begin{equation}
        \begin{split}
            \parallel \nabla f_{\epsilon} \parallel_{L^{2}} &\leq \frac{C}{\sqrt{\epsilon}} \parallel f \parallel_{H^{1}}\\
            \parallel f_{\epsilon} \parallel_{L^{2}} &\leq (1+C\sqrt{\epsilon}) \parallel f \parallel_{L^{2}},
        \end{split}
    \end{equation}
    where $C$ is a constant that depends only on $\Omega$. 
\end{itemize}

In addition, if $\text{tr}(f) \geq 0,$ then $f_{\epsilon}$ is non negative in $\Omega_{\epsilon} \setminus \Omega.$
\end{lemma}

With the the help of the last lemma, we can prove a proposition, which will be needed in the proof of the concentration inequality. 

\begin{proposition}\label{reducingH-1tocomactsets}
Let $\nu$ be a measure such that $\|\nu\|_{H^{-1}} < \infty$ and assume that there exists a compact set $K$ such that $\nu$ is nonpositive or nonnegative outside of $K,$ and the boundary of $K$ has $C^{2}$ regularity. Then there exists a compact set $K_{1}$ which contains $K,$ and a constant $c$ such that 
\begin{equation}
    \parallel \nu \parallel_{H^{-1}} \geq c \parallel \nu|_{K_{1}} \parallel_{H^{-1}(K_{1})}.
\end{equation}
Furthermore, $c$ and $K_{1}$ depend only on $K.$
\end{proposition}
The proof is found in the appendix. This proposition is restated as Proposition \ref{localizingH-1norm}.

With the last propositions, we can prove Theorem \ref{rateofconvergencethermalequilibrium}:
\begin{proof}
(Of Theorem \ref{rateofconvergencethermalequilibrium}). Let $k>0,$ we have that 
\begin{equation}
\begin{split}
&\mathbf{P}_{N, \beta} \left( \parallel  \text{emp}_{N}-\mu_{\beta} \parallel_{BL} \leq \frac{k}{N^{\frac{1}{d}}} \right)=\\
&1-\mathbf{P}_{N, \beta} \left( \parallel  \text{emp}_{N}-\mu_{\beta} \parallel_{BL} > \frac{k}{N^{\frac{1}{d}}} \right) \geq \\
& 1- \frac{1}{K_{N,\beta}} \exp \left( -\frac{N^{2}\beta}{2} \inf_{\mu \in \left( B \left( \mu_{\beta}, \frac{k}{N^{\frac{1}{d}}} \right) \right)^{C} \cap \mathcal{P}_{N}} G\left( \mu-\mu_{\beta}, \mu-\mu_{\beta} \right) \right).
\end{split}
\end{equation}
Using Propositions \ref{H-1iselectrostaticenergy} and \ref{approximationbysmoothfunctions}, we have that in dimension $3$ or higher,
\begin{equation}
\begin{split}
   &\inf_{\mu \in \left( B \left( \mu_{\beta}, \frac{k}{N^{\frac{1}{d}}} \right) \right)^{C} \cap \mathcal{P}_{N}} G\left( \mu-\mu_{\beta}, \mu-\mu_{\beta} \right) \geq \\
   &\inf_{\mu \in \left( B \left( \mu_{\beta}, \frac{(k-c_{1})_{+}}{N^{\frac{1}{d}}} \right) \right)^{C} \cap \mathcal{S}} G\left( \mu-\mu_{\beta}, \mu-\mu_{\beta} \right) - \frac{C}{N^\frac{2}{d}} = \\
   &  \inf_{\mu \in \left( B \left( \mu_{\beta}, \frac{(k-c_{1})_{+}}{N^{\frac{1}{d}}} \right) \right)^{C} \cap \mathcal{S}} \parallel \mu-\mu_{\beta} \parallel_{H^{-1}}^{2} - \frac{C}{N^\frac{2}{d}}.
\end{split}
\end{equation}

In dimension 2, using Propositions \ref{boundonpartitionfunction} and \ref{approximationbysmoothfunctions} we have that
\begin{equation}
    \begin{split}
       & \   -\log(K_{N,\beta})-\frac{N^{2}\beta}{2} \inf_{\mu \in \left( B \left( \mu_{\beta}, \frac{k}{N^{\frac{1}{d}}} \right) \right)^{C} \cap \mathcal{P}_{N}} G\left( \mu-\mu_{\beta}, \mu-\mu_{\beta} \right) \\
        &\leq -\log(K_{N,\beta}) -\frac{N^{2}\beta}{2} \left( \inf_{\mu \in \left( B \left( \mu_{\beta}, \frac{(k-c_{1})_{+}}{N^{\frac{1}{d}}} \right) \right)^{C} \cap \mathcal{S}} G \left(\mu-\mu_{\beta}, \mu-\mu_{\beta} \right) - \frac{C}{N}-\frac{1}{2 N} \log(N) \right) \\
        &= \frac{N^{2} \beta}{2} \inf_{\mu \in \left( B \left( \mu_{\beta}, \frac{(k-c_{1})_{+}}{N^{\frac{1}{d}}} \right) \right)^{C} \cap \mathcal{S}} \parallel \mu-\mu_{\beta} \parallel_{H^{-1}}^{2} - \frac{C}{N}.
    \end{split}
\end{equation}

Now we use Propositions \ref{localizingBLnorm} and \ref{reducingH-1tocomactsets} to get a lower bound on the expression on the last line. Let $\nu_{N}$ be a sequence of absolutely continuous probability measures such that 
\begin{equation}
    \parallel \nu_{N}- \mu_{\beta} \parallel_{BL} \geq \frac{(k-c_{1})_{+}}{N^{\frac{1}{d}}}.
\end{equation}
Then we claim that there exists $c_{2}>0$ such that
\begin{equation}
    \parallel \nu_{N}-\mu_{\beta} \parallel_{H^{-1}} \geq c_{2}  \frac{(k-c_{1})_{+}}{N^{\frac{1}{d}}}.
\end{equation}

In order to show this, note that using properties $3$ and $4$ of $V$ and $\beta$ (equations \eqref{lastpropertyofbeta} and \eqref{lastpropertyofbeta2}) there exists a compact set $K_{1}$ such that 
\begin{equation}
    \mu_{\beta}^{K_{1}}:=\frac{1}{\int_{K_{1}} \mu_{\beta} \, dx}\mu_{\beta}|_{K_{1}}
\end{equation}

satisfies that 
\begin{equation}\label{smallnessofleftover}
    \begin{split}
        \parallel \mu_{\beta}^{K_{1}}- \mu_{\beta} \parallel_{BL} &\ll \frac{1}{N^{\frac{1}{d}}}\\
        \mathcal{E}\left( \mu_{\beta}-\mu_{\beta}^{K_{1}} \right) &\ll \frac{1}{N^{\frac{1}{d}}}.
    \end{split}
\end{equation}

Furthermore, equation \eqref{smallnessofleftover} also holds for any compact set that contains $K_{1}.$ By equation \eqref{reducingBLtocpmtactsets}, we also know that there exists a compact set $K_{2}$ such that 
\begin{equation}\label{referenceline814}
     \parallel (\mu_{\beta}^{K_{2}}- \nu_{N}) \mathbf{1}_{K_{2}} \parallel_{BL} \geq \frac{1}{4} \parallel \mu_{\beta}- \nu_{N} \parallel_{BL},
\end{equation}
furthermore, \eqref{referenceline814} also holds for any compact set that contains $K_{2}$. Also by hypothesis
\begin{equation}
     \parallel \mu_{\beta}- \nu_{N} \parallel_{BL} \geq  \frac{(k-c_{1})_{+}}{N^{\frac{1}{d}}}.
\end{equation}

Let $K_{3}=K_{1} \bigcup K_{2},$ then $\nu_{N}-\mu_{\beta}^{K_{3}}$ is non negative outside of $K_{3},$ and therefore by Proposition \ref{reducingH-1tocomactsets} there exists a compact set $K$ which contains $K_{3}$ and a constant $c_{4}$ such that 
\begin{equation}
    \parallel (\mu_{\beta}^{K}- \nu_{N})|_{K} \parallel_{H^{-1}(K)} \leq c_{4} \parallel \mu_{\beta}^{K}- \nu_{N} \parallel_{H^{-1}}.
\end{equation}

Putting everything together, we get that
\begin{equation}
    \begin{split}
        \parallel \mu_{\beta}^{K}- \nu_{N} \parallel_{H^{-1}} &\geq c  \parallel (\mu_{\beta}^{K}- \nu_{N})|_{K} \parallel_{H^{-1}(K)} \\
        &\geq c  \parallel (\mu_{\beta}^{K}- \nu_{N}) \mathbf{1}_{K} \parallel_{BL} \\
        &\geq c  \parallel \mu_{\beta}- \nu_{N} \parallel_{BL} \\
        &\geq  c_{2}^{*}\frac{(k-c_{1})_{+}}{N^{\frac{1}{d}}}.
    \end{split}
\end{equation}
Lastly, we have that if 
\begin{equation}
    \| \nu_{N} - \mu_{\beta} \|_{BL} \geq \frac{k}{N^{\frac{1}{d}}},
\end{equation}
then for some $c_{2}^{*} \in \mathbf{R}^{+} $ and $N$ big enough
\begin{equation}
\begin{split}
    \mathcal{E}(\nu_{N}-\mu_{\beta}) & \geq    c_{6}\mathcal{E}(\nu_{N}-\mu_{\beta}^{K})-c_{7}\mathcal{E}(\mu_{\beta}-\mu_{\beta}^{K})\\
    &=c_{6}\|\nu_{N}-\mu_{\beta}^{K}\|_{H^{-1}}-c_{7}\mathcal{E}(\mu_{\beta}-\mu_{\beta}^{K})\\
    &\geq c_{2}^{*}\frac{(k-c_{5})_{+}}{N^{\frac{1}{d}}} - c_{7}\mathcal{E}(\mu_{\beta}-\mu_{\beta}^{K}).
\end{split}
\end{equation}
Therefore for $N$ big enough, we have that 
\begin{equation}
    \begin{split}
        \mathbf{P}_{N, \beta} \left( \parallel  \text{emp}_{N}-\mu_{\beta} \parallel_{BL} \leq\frac{k}{N^{\frac{1}{d}}} \right) &\geq 1- \exp\left(-\frac{1}{2} N^{2-\frac{2}{d}}\beta\left( c_{1}(k-c_{2})_{+}^{2}-c_{3} \right) \right)\\
        &\to 1,
    \end{split}
\end{equation}
where convergence happens for all $k> \sqrt{\frac{c_{3}}{c_{1}}}+c_{2}.$

\end{proof}

As a consequence of our methods, we obtain the following theorems relating the bounded Lipschitz norm to the $H^{-1}$ norm (electrostatic energy). 
\begin{remark}\label{theofromBLtoH-1(1)}
Let $\mu$ be a measure of bounded variation. Assume further that if $\mu$ is defined on $\mathbf{R}^{2}$ then $\mu$ has mean $0$, and assume that there exists a compact set $K$ such that $\mu$ has a definite sign outside of $K,$ and $K$ has $C^{2}$ regularity. Then there exists a constant $k,$ and a compact set $K_{2}$ which depend only on $K$ such that 
\begin{equation}
    \parallel \mu|_{K_{2}} \parallel_{BL}^{2} \leq k    \mathcal{E}(\mu).
\end{equation}
\end{remark}

\begin{proof}
By Proposition \ref{reducingH-1tocomactsets}, we have that for some $k,$
\begin{equation}
    \|\mu \|_{H^{-1}} \geq k \|\mu|_{K_{1}} \|_{H^{-1}(K_{1})}.
\end{equation}
Together with Proposition \ref{fromBLtoH-1}, this implies 
\begin{equation}
    \begin{split}
         \|\mu \|_{H^{-1}} &\geq k \|\mu|_{K_{1}} \|_{H^{-1}(K_{1})}\\
         &\geq k \|\mu|_{K_{1}} \|_{BL}.
    \end{split}
\end{equation}
Using Proposition \ref{H-1iselectrostaticenergy} we can conclude.
\end{proof}

\begin{remark}\label{theofromBLtoH-1(2)}
Let $\mu$ be a measure of bounded variation. Assume further that if $\mu$ is defined on $\mathbf{R}^{2}$ then $\mu$ has mean $0$ and assume that there exist compacts sets $K_{1}, K_{2}$ such that $\mu|_{K_{2} \setminus K_{1}}$ has a density which is in $L^{2},$ and $K_{1}$ has $C^{2}$ regularity. Then there exists a constant $k,$ which depends on $K_{1}, K_{2},$ and $ \parallel \mu|_{K_{2} \setminus K_{1}} \parallel_{L^{2}}$ such that 
\begin{equation}
    \parallel \mu|_{K_{1}} \parallel_{BL}^{2} \leq k    \mathcal{E}(\mu).
\end{equation}
\end{remark}

\begin{proof}
By Proposition \ref{fromBLtoH-1(2)}, we have that
\begin{equation}
    \|\mu \|_{H^{-1}} \geq k \|\mu|_{K_{2}} \|_{H^{-1}(K_{2})}.
\end{equation}
Together with Proposition \ref{fromBLtoH-1}, this implies 
\begin{equation}
    \begin{split}
         \|\mu \|_{H^{-1}} &\geq k \|\mu|_{K_{2}} \|_{H^{-1}(K_{2})}\\
         &\geq k \|\mu|_{K_{2}} \|_{BL}.
    \end{split}
\end{equation}
Using proposition \ref{H-1iselectrostaticenergy} we can conclude.
\end{proof}

\begin{remark}
Clearly there is no positive constant $k$ such that
\begin{equation}
    \parallel \mu \parallel_{BL} \leq k    \parallel \mu \parallel_{H^{-1}}.
\end{equation}
\end{remark}

\section{Appendix 1}

This appendix is devoted to proving the following proposition:
\begin{proposition}\label{approximationbysmoothfunctions2}
Let $\lambda=\frac{1}{d},$ and let $A_{N}=\mathcal{P}(\mathbf{R}^{d})\setminus B\left( \mu_{\beta}, \frac{k_{1}}{N^{\lambda}} \right),$ where $\mathcal{P}(\mathbf{R}^{d})$ is the set of probability measures on $\mathbf{R}^{d}.$ Let $\mathcal{S}$ denote the set of absolutely continuous probability measures, then there exists a constant $C$ (which depends only on $\| \mu_{V} \|_{L^{\infty}}$) such that, if $d \geq 3$ then
\begin{equation}
    \inf_{\mu \in A_{N} \cap \mathcal{P}_{N} } G \left(\mu-\mu_{\beta}, \mu-\mu_{\beta} \right) \geq \inf_{\mu \in B_{N}\cap  \mathcal{S}} G \left(\mu-\mu_{\beta}, \mu-\mu_{\beta} \right) - \frac{C}{N^{\frac{2}{d}}},
\end{equation}
where 
\begin{equation}
    B_{N}=\mathcal{P}(\mathbf{R}^{d})\setminus B\left( \mu_{\beta}, \frac{(k_{1}-k_{2})_{+}}{N^{\lambda}} \right),
\end{equation}
for some absolute constant $k_{2},$ where
\begin{equation}
    (x)_{+}=
    \begin{cases}
    x \text{ if } x \geq 0\\
    0 \text{ o.w.}
    \end{cases}
\end{equation}

If $d=2$ then
\begin{equation}
    \inf_{\mu \in A_{N} \cap \mathcal{P}_{N} } G \left(\mu-\mu_{\beta}, \mu-\mu_{\beta} \right) \geq \inf_{\mu \in B_{N}\cap  \mathcal{S}} G \left(\mu-\mu_{\beta}, \mu-\mu_{\beta} \right) - \frac{C}{N}-\frac{1}{2 N} \log(N).
\end{equation}
\end{proposition}

This Proposition was already stated as Proposition  \ref{approximationbysmoothfunctions}. This section uses ideas very similar to ones found in \cite{rougerie2016higher} and \cite{hardin2018large}. We begin by recalling a few facts about the Coulomb Kernel. These can be found in \cite{chafai2018concentration}, or deduced using superharmonicity.

\begin{lemma}
Let $\lambda_{R}$ be the uniform probability measure on a ball of radius $R$ centered at $0$, then for every $x \in \mathbf{R}^{d},$ we have that 
\begin{equation}\label{superharmoniceq1}
   \int_{\mathbf{R}^{d}} g(x+u) \lambda_{R}(u) \, du \leq g(x)
\end{equation}
and also that 
\begin{equation}\label{superharmoniceq2}
    \iint_{\mathbf{R}^{d} \times \mathbf{R}^{d}} g(x+u-v) \lambda_{R}(u) \lambda_{R}(v) \, du \, dv \leq g(x).
\end{equation}
Furthermore, eqs \eqref{superharmoniceq1} and \eqref{superharmoniceq2} become an equality if $|x|>R.$
\end{lemma}
The next lemma can also be found in \cite{chafai2018concentration} (or verified by direct computation).
\begin{lemma}
Let $\lambda_{R}$ be the uniform measure on the ball of radius $R,$ then for $d \geq 3$
\begin{equation}
    G(\lambda_{R}, \lambda_{R})= \frac{g(R)}{g(1)}    G(\lambda_{1}, \lambda_{1}).
\end{equation}
For $d=2$ we have that
\begin{equation}
     G(\lambda_{R}, \lambda_{R})=g(R)+G(\lambda_{1}, \lambda_{1}).
\end{equation}
\end{lemma}
We need one more lemma before embarking on the proof of \ref{approximationbysmoothfunctions2}. 

\begin{lemma}
Let $\left\{x_{i} \right\}_{i=1}^{N} \in \mathbf{R}^{d},$  let $P=\frac{1}{N}\sum_{i=1}^{N} \delta_{x_{i}}$ and $P_{\epsilon}=P \ast \lambda_{\epsilon}.$ Then, if $d \geq 3$,
\begin{equation}\label{approximatingenergybysmoothfunctions}
    \frac{1}{N^{2}}\sum_{i\neq j} g(x_{i}-x_{j}) \geq G\left( P_{\epsilon}, P_{\epsilon} \right)-\frac{1}{N g(1)} g(\epsilon)G(\lambda_{1}, \lambda_{1}).
\end{equation}
Furthermore, eq. \eqref{approximatingenergybysmoothfunctions} is an equality if $\epsilon \leq \min\left\{ |x_{i}-x_{j}|\right\}.$
\end{lemma}
\begin{proof}
The proof is found in \cite{serfaty2020mean} and in \cite{chafai2018concentration}.
\end{proof}

We now give the proof of Proposition \ref{approximationbysmoothfunctions2}:
\begin{proof}
(Of Proposition \ref{approximationbysmoothfunctions2})

Our goal is to prove that 
\begin{equation}
    G\left( P-\mu_{\beta}, P-\mu_{\beta} \right)  \geq   G\left( P_{\epsilon}-\mu_{\beta}, P_{\epsilon} -\mu_{\beta} \right)  -\frac{C}{N^\frac{2}{d}}
\end{equation}
for some constant $C$, and the right choice of $\epsilon$ (which will depend on $N$). 

Expanding, we get
\begin{equation}
    G\left( P_{\epsilon}-\mu_{\beta}, P_{\epsilon} -\mu_{\beta} \right)=  G\left( P_{\epsilon}, P_{\epsilon}\right)- 2G\left( P_{\epsilon},\mu_{\beta} \right)+ G\left(\mu_{\beta},\mu_{\beta} \right).
\end{equation}

Using equation \eqref{approximatingenergybysmoothfunctions} we  immediately get that, if $d \geq 3$ then
\begin{equation}
    G\left( P_{\epsilon}, P_{\epsilon}\right) \geq G\left( P, P\right)+\frac{C g(\epsilon)}{N }.
\end{equation}

Our goal is now to get an upper bound for $G\left( P_{\epsilon},\mu_{\beta} \right)$ in terms of $G\left( P,\mu_{\beta} \right),$ which we do using superharmonicity. For any $\epsilon>0,$ we begin by writing, for $d \geq 3$ (recall the abuse of notation of not distinguishing between a measure and its density):
\begin{equation}\label{firstboundodsmearing}
\begin{split}
    &G(P, \mu_{\beta})=\\
    & \frac{1}{N}\sum_{i=1}^{N} \int_{\mathbf{R}^{d}} g(y-x_{i}) \mu_{\beta}(y) \, dy = \\
    & \frac{1}{N}\sum_{i=1}^{N}  \int_{\mathbf{R}^{d} \setminus B(x_{i}, \epsilon)} g(y-x_{i}) \mu_{\beta}(y) \, dy+  \int_{ B(x_{i}, \epsilon)} g(y-x_{i}) \mu_{\beta}(y) \, dy =\\
     & \frac{1}{N}\sum_{i=1}^{N}  \iint_{y \in \mathbf{R}^{d} \setminus B(x_{i}, \epsilon)\,  ,s\in \mathbf{R}^{d}} g(y-x_{i}+s) \lambda_{\epsilon}(s) \mu_{\beta}(y) \, ds \, dy\ + \\
     &\quad \quad \quad \quad \int_{ B(x_{i}, \epsilon)} g(y-x_{i}) \mu_{\beta}(y) \, dy =\\
     & \frac{1}{N}\sum_{i=1}^{N}  \iint_{y \in \mathbf{R}^{d} \setminus B(x_{i}, \epsilon)\,  ,s\in \mathbf{R}^{d}} g(y-x_{i}+s) \lambda_{\epsilon}(s) \mu_{\beta}(y) \, ds \, dy\ +\\
     &\quad \quad \quad \quad \int_{y \in  B(x_{i}, \epsilon), \, s \in \mathbf{R}^{d}} g(y-x_{i}+s) \mu_{\beta}(y) \, dy \, d(\delta_{0}+\lambda_{\epsilon}-\lambda_{\epsilon})(s) = \\
     &\frac{1}{N} \sum_{i=1}^{N}  \iint_{\mathbf{R}^{d} \times \mathbf{R}^{d}} g(y-x_{i}+s) \lambda_{\epsilon}(s) \mu_{\beta}(y) \, ds \, dy\ +\\
       &\quad \quad \quad \quad \int_{y \in  B(x_{i}, \epsilon), \, s \in \mathbf{R}^{d}} g(y-x_{i}+s) \mu_{\beta}(y) \, dy \, d(\delta_{0}-\lambda_{\epsilon})(s) = \\
     & G(P_{\epsilon}, \mu_{\beta}) + \frac{1}{N} \sum_{i=1}^{N}  \int_{y \in  B(x_{i}, \epsilon)} \left[ -\int_{s \in \mathbf{R}^{d}} g(y-x_{i}+s)\lambda_{\epsilon}(s) \, ds+g(x_{i}-y) \right] \mu_{\beta}(y) \, dy \leq \\
     & G(P_{\epsilon}, \mu_{\beta}) + \frac{1}{N} \sum_{i=1}^{N}  \int_{y \in  B(x_{i}, \epsilon)} g(x_{i}-y) \mu_{\beta}(y) \, dy. 
\end{split}
\end{equation}

Since $\mu_{\beta}$ is uniformly bounded in $L^{\infty}$ by Proposition \ref{muthetaboundeduniformly}, we then have that, for $d \geq 3$, 
\begin{equation}
    \begin{split}
        G(P_{\epsilon}, \mu_{\beta}) &\geq G(P, \mu_{\beta}) - \max_{\beta} \{\parallel \mu_{\beta} \parallel_{L^{\infty}} \} \int_{y \in  B(0, \epsilon)} g(y) \, dy \\
         &=G(P, \mu_{\beta}) - C \epsilon^{2},
    \end{split}
\end{equation}
where $C$ depends on $\| \mu_{V} \|_{L^{\infty}}$

In the last equation, we have used that, if $d\geq 3$ then
\begin{equation}
    \begin{split}
        \int_{y \in  B(0, \epsilon)} g(y) \, dy &=c_{d}\int_{0}^{\epsilon} \frac{r^{d-1}}{r^{d-2}}\\
        &=C_{d} \epsilon^{2}
    \end{split}
\end{equation}
where $c_{d}$ depends on $d.$

In conclusion, we have, for $d\geq 3,$ 
\begin{equation}
    G\left( P-\mu_{\beta}, P -\mu_{\beta} \right) \geq G\left( P_{\epsilon}-\mu_{\beta}, P_{\epsilon} -\mu_{\beta} \right)-C\frac{g(\epsilon)}{N }-C_{d} \epsilon^{2}.
\end{equation}
Taking $\epsilon=\frac{1}{N^{\frac{1}{d}}},$ we have that, for $d \geq 3$ 
\begin{equation}
    \begin{split}
    G\left( P-\mu_{\beta}, P -\mu_{\beta} \right) &\geq G\left( P_{\frac{1}{N^{\frac{1}{d}}}}-\mu_{\beta}, P_{\frac{1}{N^{\frac{1}{d}}}} -\mu_{\beta} \right)-C\frac{g\left(\frac{1}{N^{\frac{1}{d}}}\right)}{N }-C_{d} \left( \frac{1}{N^{\frac{1}{d}}}\right)^{2} \\
    &=G\left( P_{\frac{1}{N^{\frac{1}{d}}}}-\mu_{\beta}, P_{\frac{1}{N^{\frac{1}{d}}}} -\mu_{\beta} \right) - C N^{-\frac{2}{d}}.
    \end{split}
\end{equation}

To deal with the case $d=2,$ we start from the penultimate line of equation \eqref{firstboundodsmearing} (note that until this point, equation \eqref{firstboundodsmearing} is valid for $d=2$) and proceed as in lemma 3.5 of \cite{rougerie2016higher}:
\begin{equation}
    \begin{split}
      & \    \left| \int_{y \in  B(x_{i}, \epsilon)} \left[ -\int_{s \in \mathbf{R}^{d}} g(y-x_{i}+s)\lambda_{\epsilon}(s) \, ds+g(x_{i}-y) \right] \mu_{\beta}(y) \, dy  \right| \\
       &\leq \|\mu_{\beta}\|_{L^{\infty}} \int_{y \in  B(x_{i}, \epsilon)} \left| \int_{s \in \mathbf{R}^{d}} \log|y-x_{i}+s|\lambda_{\epsilon}(s) \, ds-\log|x_{i}-y| \right|  \, dy \\
       &=  \|\mu_{\beta}\|_{L^{\infty}} \epsilon^{2} \int_{y \in  B(0, 1)} \left| \int_{s \in \mathbf{R}^{d}} \log|y+s|\lambda_{1}(s) \, ds-\log|y| \right|  \, dy \\
       &\leq C \|\mu_{\beta}\|_{L^{\infty}} \epsilon^{2}.
    \end{split}
\end{equation}
In conclusion, we have, for $d= 2,$ 
\begin{equation}
    G\left( P-\mu_{\beta}, P -\mu_{\beta} \right) \geq G\left( P_{\epsilon}-\mu_{\beta}, P_{\epsilon} -\mu_{\beta} \right)-g(\epsilon)\frac{1}{N }-C_{1} \epsilon^{2}-c_{2}N.
\end{equation}
Taking $\epsilon=\frac{1}{N^{\frac{1}{2}}},$ we have that
\begin{equation}
    G\left( P-\mu_{\beta}, P -\mu_{\beta} \right) \geq \inf_{\mu \in B_{N}\cap  \mathcal{S}}G\left( P_{\frac{1}{N^{\frac{1}{2}}}}-\mu_{\beta}, P_{\frac{1}{N^{\frac{1}{2}}}} -\mu_{\beta} \right) - \frac{C}{N} - \frac{1}{2N}\log N,
\end{equation}
where $C$ depends on $\| \mu_{V} \|_{L^{\infty}}.$

In order to conclude, we now only need to show that
\begin{equation}
    \parallel P-P_{N^{-\frac{1}{d}}} \parallel_{BL} \leq \frac{k}{N^{\frac{1}{d}}}
\end{equation}
for any $d \geq 2$, where $k$ depends only on $d$. The reason is elementary. Let $\varphi \in W^{1, \infty}$ be such that
\begin{equation}
    \| \varphi \|_{W^{1,\infty}}=1.
\end{equation}
Then
\begin{equation}
\left| \int_{\mathbf{R}^{d}} \varphi \,  d(P-P_{\epsilon}) \right| \leq \frac{1}{N} \sum_{i=1}^{N} \left| \strokedint_{B(x_{i}, \epsilon)} \varphi(x)dx -\varphi(x_{i}) \right|.
\end{equation}
Since $\|\varphi\|_{W^{1,\infty}}=1,$ we have, for $x \in B(x_{i}, \epsilon)$ that
\begin{equation}
    \| \varphi(x)-\varphi(x_{i})\| \leq \epsilon,
\end{equation}
therefore
\begin{equation}
    \left| \strokedint_{B(x_{i}, \epsilon)} \varphi(x)dx -\varphi(x_{i}) \right| \leq \epsilon,
\end{equation}
and 
\begin{equation}
\left| \int_{\mathbf{R}^{d}} \varphi \,  d(P-P_{\epsilon}) \right| \leq \epsilon.    
\end{equation}

\end{proof}

\section{Appendix 2}

This appendix is devoted to proving results the related to $H^{-1}(K)$ and $H^{-1}$ norms that we used in the paper. The first result we need is 
\begin{lemma}\label{extensionlemma2}
Let $\Omega \subset \mathbf{R}^{d}$ be an open bounded set with a $C^{2}$ boundary, and let $f \in H^{1}({\Omega}).$ 
Let 
\begin{equation}
    \epsilon_{*}=\sup \{ \epsilon>0| x \mapsto x+\epsilon\nu(x) \text{ is a diffeomorphism for all } |\delta|<\epsilon \},
\end{equation}
where $\nu(x)$ is the unit normal to $\partial \Omega$ at $x.$ Then for every $0<\epsilon<\epsilon_{*}$ there exists $f_{\epsilon} \in H^{1}(\mathbf{R}^{d})$ such that 
\begin{itemize}
    \item[$\bullet$] The restriction satisfies $f_{\epsilon}|_{\Omega}=f.$
    \item[$\bullet$] The support satisfies $\mbox{supp}(f_{\epsilon}) \subset \overline{\Omega}_{\epsilon},$ where
    \begin{equation}
        \Omega_{\epsilon}=\{ x \in \mathbf{R}^{d}| d(x,\Omega) < \epsilon\}.
    \end{equation}
    \item[$\bullet$] We have control of the norms:
    \begin{equation}
        \begin{split}
            \parallel \nabla f_{\epsilon} \parallel_{L^{2}} &\leq \frac{C}{\sqrt{\epsilon}} \parallel f \parallel_{H^{1}}\\
            \parallel f_{\epsilon} \parallel_{L^{2}} &\leq (1+k\sqrt{\epsilon}) \parallel f \parallel_{L^{2}},
        \end{split}
    \end{equation}
    where $C$ and $k$ are constants that depend only on $\Omega$. 
\end{itemize}

In addition, if $\text{tr}(f) \geq 0,$ then $f_{\epsilon}$ is non negative in $\Omega_{\epsilon} \setminus \Omega.$
\end{lemma}

\begin{proof}
Step 1. Let $x \in \partial \Omega.$ We will use the notation 
\begin{equation}
    x=(\underline{x}, x_{d}),
\end{equation}
where $\underline{x} \in \mathbf{R}^{d-1}$ and $x_{d} \in \mathbf{R}.$ Assume for now that $\partial \Omega$ is flat near $x.$ In other words, that there exists some $\delta>0$ such that
\begin{equation}
    B(x,\delta) \bigcap \partial \Omega = \{ y| y_{d}=0\} \bigcap B(x,\delta),
\end{equation}
and in addition $\nu = (0,0,...1)$.
Let
\begin{equation}
    \underline{B}(x,\delta)= \{ y| y_{d}=0\} \bigcap B(x,\delta).
\end{equation}
 Let $\alpha>0$ be such that 
\begin{equation}
     \underline{B}(x,\delta) \times (-\alpha, 0) \subset \Omega.
\end{equation}
Note that $\alpha$ exists since by hypothesis $\partial \Omega$ is $C^{2}.$ Define a function $\varphi: \underline{B}(x,\delta) \times (0,\alpha) \to \mathbf{R}$ as 
\begin{equation}
    \varphi(\underline{y}, y_{d})= f(\underline{y}, -y_{d}).
\end{equation}
Let $\mu\in C^{\infty}( [0,\alpha], \mathbf{R}^{+})$ be such that $\mu(0)=1,\    \mu(\alpha)=0,$ and $\mu$ is decreasing. Consider now $\widehat{\varphi}: \underline{B}(x,\delta) \times (0,\alpha) \to \mathbf{R}$ defined as 
\begin{equation}
    \widehat{\varphi}(\underline{y}, y_{d})=\varphi(\underline{y}, y_{d}) \mu(y_{d}).
\end{equation}
Lastly, define the function $\varphi_{\epsilon}: \underline{B}(x,\delta) \times (0,\epsilon) \to \mathbf{R}$ as
\begin{equation}
    \varphi_{\epsilon}(\underline{y}, y_{d})=\widehat{\varphi}\left(\underline{y}, \frac{\alpha}{\epsilon}y_{d}\right).
\end{equation}
We immediately get the estimates
\begin{equation}
    \begin{split}
        \parallel \varphi_{\epsilon} \parallel_{L^{2}} &= \sqrt{\frac{\epsilon}{\alpha}}  \parallel \widehat{\varphi} \parallel_{L^{2}}\\
        &\leq \sqrt{\frac{\epsilon}{\alpha}}  \parallel {\varphi} \parallel_{L^{2}}\\
        &\leq \sqrt{\frac{\epsilon}{\alpha}}  \parallel f \parallel_{L^{2}}.
    \end{split}
\end{equation}
We also have the estimates
\begin{equation}
    \begin{split}
        \parallel \nabla \varphi_{\epsilon} \parallel_{L^{2}} &\leq C \max\left( \frac{\sqrt{\alpha}}{\sqrt{\epsilon}}, 1 \right)  \parallel \nabla \widehat{\varphi} \parallel_{L^{2}}\\
        &\leq C \max\left( \frac{\sqrt{\alpha}}{\sqrt{\epsilon}}, 1 \right) \left( \parallel \nabla({\varphi})\mu \|_{L^{2}} + \| \varphi \frac{d}{dx}\mu \parallel_{L^{2}} \right)\\
        &\leq C \max\left( \frac{\sqrt{\alpha}}{\sqrt{\epsilon}}, 1 \right)  \parallel f \parallel_{H^{1}},
    \end{split}
\end{equation}
where $C$ depends only on $\Omega.$

Lastly, if $\text{tr}(f) \geq 0,$ consider the function 
\begin{equation}
    M\varphi_{\epsilon}=\max \left( \varphi_{\epsilon}, 0 \right).
\end{equation}
Then $M\varphi_{\epsilon}$ is positive, and 
\begin{equation}
    M\varphi_{\epsilon}(\underline{y}, 0)=\varphi_{\epsilon}(\underline{y}, 0)
\end{equation}
for any $\underline{y} \in  \underline{B}(x,\delta).$
Using the identity
\begin{equation}
    M\varphi_{\epsilon} = \frac{1}{2} \left( |\varphi_{\epsilon} |+\varphi_{\epsilon}  \right),
\end{equation}
we get that
\begin{equation}
    \begin{split}
        \parallel \nabla M\varphi_{\epsilon} \parallel_{L^{2}} &\leq  \parallel \nabla \varphi_{\epsilon} \parallel_{L^{2}} \\
        \parallel  M\varphi_{\epsilon} \parallel_{L^{2}} &\leq  \parallel \varphi_{\epsilon} \parallel_{L^{2}}. 
    \end{split}
\end{equation}

Step 2. Now we turn to the general case, where $\partial \Omega$ is not necessarily locally flat. Since by assumption $\partial \Omega$ is $C^{2},$ there exist finitely many balls $B(x_{i}, \epsilon_{i})$ and $C^{2}$ diffeomorphisms $g_{i}: U_{i} \subset \mathbf{R}^{d-1} \to \mathbf{R}^{d}$ such that
\begin{equation}
    g_{i}\left( U_{i} \right)= B(x_{i}, \epsilon_{i}) \bigcap \partial \Omega.
\end{equation}
For any $\delta < \epsilon_{*}, $ we can extend $g_{i}$ to a $C^{1}$ diffeomorphism $\overline{g}_{i}: U_{i} \times (-\delta, \delta) \to P_{i}^{\delta},$ where
\begin{equation}
P_{i}^{\delta}= \{ x+s\nu(x)| x \in  B(x_{i}, \epsilon_{i}) \bigcap \partial \Omega, s \in (-\delta, \delta) \},
\end{equation}
where $\nu(x)$ is the unit normal to the point $x.$ We define $\overline{g}_{i}$ as
\begin{equation}
    \overline{g}_{i}(\underline{x}, s)= g_{i}(\underline{x})+s\nu(g_{i}(\underline{x})).
\end{equation}

Now define for any $\epsilon<\epsilon_{*}$ the function $\varphi_{\epsilon}^{i}:  U_{i} \times (-\epsilon, \epsilon)$ as in step 1, with $\alpha=\epsilon_{*}$. If $\text{tr}(f) \geq 0,$ define $M \varphi_{\epsilon}^{i}$ as in step 1. 

Define the functions $\phi_{\epsilon}^{i}$ as 
\begin{equation}
    \phi_{\epsilon}^{i}= \overline{g}_{i} \circ \varphi_{\epsilon}^{i}\circ \overline{g}_{i}^{-1}.
\end{equation}
 If $\text{tr}(f) \geq 0,$ define the functions $M\phi_{\epsilon}^{i}$ as 
\begin{equation}
    M\phi_{\epsilon}^{i}= \overline{g}_{i} \circ M\varphi_{\epsilon}^{i}\circ \overline{g}_{i}^{-1}.
\end{equation}

Lastly, take a partition of unity $q_{i}$ associated to $P_{i}^{\epsilon_{*}}.$ Define the extension $f_{\epsilon}$ as 
\begin{equation}
   f_{\epsilon}(x) =
   \begin{cases}
    f(x) \text{ if } x \in \Omega\\
    \sum_{i}(q_{i} \phi_{i}^{\epsilon})  \text{ if } x \in \bigcup P_{i}^{\delta}\\
    0 \text{ o.w.}
    \end{cases}
\end{equation}

 If $\text{tr}(f) \geq 0,$ define the extension $Mf_{\epsilon}$ as 
\begin{equation}
   Mf_{\epsilon}(x) =
   \begin{cases}
    f(x) \text{ if } x \in \Omega\\
    \sum_{i}(q_{i} M\phi_{i}^{\epsilon})  \text{ if } x \in \bigcup P_{i}^{\delta}\\
    0 \text{ o.w.}
    \end{cases}
\end{equation}

It is easy to check that $f_{\epsilon}, Mf_{\epsilon}$ saitsfy the desired properties. 
\end{proof}

Next is a proposition which is not directly related to the concentration inequality, but we include since it is needed to prove remark \ref{theofromBLtoH-1(2)}. 

\begin{proposition}\label{fromBLtoH-1(2)}
Let $\nu$ be a measure such that $\|\nu\|_{H^{-1}} < \infty$. Assume that there exist compact sets $K_{1,} K_{2}$ with $K_{1}$ properly contained in $K_{2}$ such that $\nu|_{K_{2} \setminus K_{1}} \in L^{2},$ and the boundary of $K_{1}$ is $C^{2}.$ Then there exists a constant $c,$ which depends only on $K_{1}, K_{2},$ and $\parallel \nu|_{K_{2} \setminus K_{1}} \parallel_{L^{2}}$ such that
\begin{equation}
    \parallel \nu \parallel_{H^{-1}} \geq c  \parallel \nu|_{K_{1}} \parallel_{H^{-1}(K_{1})}.
\end{equation}
\end{proposition}

\begin{proof}

First, assume that 
\begin{equation}
    \parallel \nu|_{K_{1}} \parallel_{H^{-1}(K_{1})}=1.
\end{equation}
For the general case, we can apply this result to
\begin{equation}
    \widetilde{\nu}=\frac{1}{\parallel \nu|_{K_{1}} \parallel_{H^{-1}(K_{1})}} \nu.
\end{equation}

Let $\epsilon_{*}$ be as in lema \ref{extensionlemma} for $K_{1}$, and let $\widehat{\epsilon} < \epsilon_{*}$ be such that
\begin{equation}
    K_{1}^{\widehat{\epsilon}} \subset K_{2},
\end{equation}
where
\begin{equation}
        K_{1}^{\widehat{\epsilon}}=\{ x \in \mathbf{R}^{d}| d(x,K_{1}) < \widehat{\epsilon}\}.
\end{equation}
We know that $\widehat{\epsilon}$ exists since $K_{1}$ is properly contained in $K_{2}.$

Let $\varphi \in H^{1}(K_{1})$ be such that
\begin{equation}
    \begin{split}
        \int_{K_{1}} \nu \varphi &\geq \frac{1}{2} \parallel \nu|_{K_{1}} \parallel_{H^{-1}(K_{1})}\\
        &=\frac{1}{2}. 
    \end{split}
\end{equation}
and 
\begin{equation}
    \parallel \varphi \parallel_{H^{1}}=1.
\end{equation}
For any $\epsilon< \widehat{\epsilon},$ consider an extension $\varphi_{\epsilon}$ of $\varphi$ as in lemma \ref{extensionlemma}. Note that 
\begin{equation}
    \begin{split}
       \left| \int_{K_{1}^{\epsilon} \setminus K_{1}} \nu \varphi_{\epsilon} \right| &\leq \parallel \nu|_{K_{1}^{\epsilon} \setminus K_{1}} \parallel_{L^{2}} \parallel \varphi_{\epsilon}|_{K_{1}^{\epsilon} \setminus K_{1}}
        \parallel_{L^{2}}\\
        &\leq c_{1}\sqrt{\epsilon}M,
    \end{split}
\end{equation}
where 
\begin{equation}
    M=\parallel \nu|_{K_{2} \setminus K_{1}} \parallel_{L^{2}}
\end{equation}
and $c_{1}$ depends only on $K_{1}.$
We then have that
\begin{equation}
    \begin{split}
       \left| \int_{\mathbf{R}^{d}} \nu \varphi_{\epsilon} \right|&= \left| \int_{K_{1}^{\epsilon}} \nu \varphi_{\epsilon} \right|\\
        &=\left| \int_{K_{1}} \nu \varphi+\int_{K_{1}^{\epsilon} \setminus K_{1}} \nu \varphi_{\epsilon} \right|\\
        &\geq \frac{1}{2}-c_{1}M\sqrt{\epsilon}.
    \end{split}
\end{equation}
And hence, we have
\begin{equation}
    \begin{split}
         \parallel \nu \parallel_{H^{-1}}&\geq \frac{1}{\parallel \nabla \varphi_{\epsilon} \parallel_{L^{2}}} \int_{\mathbf{R}^{d}} \nu \varphi_{\epsilon}\\
         &\geq c_{2}\sqrt{\epsilon}\left(\frac{1}{2}-c_{1}M\sqrt{\epsilon}\right),
    \end{split}
\end{equation}
where $c_{2}$ depends on $K_{1}, K_{2}, \| \nu|_{K_{2} \setminus K_{1}} \|_{L^{2}}.$ Letting $\epsilon$ be small enough (for example $\sqrt{\epsilon}=\frac{1}{4 M c_{1}}$), we obtain the conclusion. 
\end{proof}

We now prove the main proposition used in the proof of the concentration inequality, which is a restatement of Proposition \ref{reducingH-1tocomactsets}:
\begin{proposition}\label{localizingH-1norm}
Let $\nu$ be a measure such that $\|\nu\|_{H^{-1}}$ and assume that there exists a compact set $K$ such that $\nu$ is nonpositive or nonnegative outside of $K,$ and the boundary of $K$ has $C^{2}$ regularity. Then there exists a compact set $K_{1}$ which contains $K,$ and a constant $c$ such that 
\begin{equation}
    \parallel \nu \parallel_{H^{-1}} \geq c \parallel \nu|_{K_{1}} \parallel_{H^{-1}(K_{1})}.
\end{equation}
Furthermore, $c$ and $K_{1}$ depend only on $K.$
\end{proposition} 

The proof of proposition \ref{localizingH-1norm} depends on the following claim.
\begin{claim}
Let $\nu$ be a measure such that $\|\nu\|_{H^{-1}} < \infty$ and assume that there exists a compact set $K$ such that $\nu$ is nonpositive or nonnegative outside of $K.$ Let 
\begin{equation}
    K_{\epsilon}=\{x \in \mathbf{R}^{d} | d(x,K) \leq \epsilon \}
\end{equation}

 Then there exists a compact set $K_{2}$ which contains $K,$ a constant $c_{2},$ and a function $\phi \in H^{1}(K_{2})$ such that 
\begin{equation}
    \int_{K_{2}} \nu \phi \geq c_{2} \parallel \nu|_{K_{2}} \parallel_{H^{-1}(K_{2})},
\end{equation}
$\parallel \phi \parallel_{H^{1}}=1$, and $\text{tr}(\phi)\geq 0.$ Furthermore, $c_{2}$ depends only on $K$ and $\epsilon.$
\end{claim}

\begin{proof}
Since
\begin{equation}
    \parallel \nu \parallel_{H^{-1}}= \parallel -\nu \parallel_{H^{-1}},
\end{equation}
we assume without loss of generality that $\nu$ is positive outside of $K.$
Note that 
\begin{equation}
    \parallel \nu|_{K_{\epsilon}} \parallel_{H^{-1}(K_{\epsilon})} < \infty,
\end{equation}
and hence, there exists some $\varphi \in H^{1}(K_{\epsilon})$ such that
\begin{equation}
    \parallel \varphi \parallel_{H^{1} } =1
\end{equation}
and 
\begin{equation}
    \int_{K_{\epsilon}} \nu \varphi \geq \frac{1}{2} \parallel \nu|_{K_{\epsilon}} \parallel_{H^{-1}(K_{\epsilon})}.
\end{equation}
Consider now $\overline{\varphi}=\varphi|_{K}.$ By the extension lemma, there exists an extension $\widetilde{\varphi}$ of $\overline{\varphi}$ such that 
\begin{equation}
    \text{supp} \left( \widetilde{\varphi} \right) \subset K_{\epsilon}
\end{equation}
and 
\begin{equation}
    \parallel \widetilde{\varphi} \parallel_{H^{1}} \leq C_{\epsilon},
\end{equation}
since 
\begin{equation}
    \parallel \varphi \parallel_{H^{1}(K_{\epsilon})} =1,
\end{equation}
where $C_{\epsilon}$ depends on $K$ and $\epsilon$. Note that $\text{tr}\left(\widetilde{\varphi}|_{K_{\epsilon}} \right) =0.$
Consider now 
\begin{equation}
      \phi=\max\{\varphi, \widetilde{\varphi}\}.
\end{equation}
Then since $\phi \geq \widetilde{\varphi},$ we know that $\text{tr}\left(\phi \right) \geq 0.$ We also know that 
\begin{equation}
\begin{split}
      \varphi (x)&=\phi(x)  \text{ for } x \in K \\
      \varphi (x)&\leq \phi(x)  \text{ for } x \in K_{\epsilon} \setminus K, \\
\end{split}
\end{equation}
which implies
\begin{equation}
    \int_{K_{\epsilon}} \nu \varphi \leq \int_{K_{\epsilon}} \nu \phi,
\end{equation}
since $\nu$ is positive outside of $K.$ Using the pointwise identity 
\begin{equation}
    \max\{\varphi, \widetilde{\varphi}\}=\frac{1}{2} \left( \varphi+ \widetilde{\varphi} + |\varphi- \widetilde{\varphi}|\right),
\end{equation}
along with triangle inequality, we get 
\begin{equation}
    \begin{split}
        \parallel \phi \parallel_{H^{1} } &\leq \parallel \varphi \parallel_{H^{1} } + \parallel \widetilde{\varphi} \parallel_{H^{1} }\\
        &\leq 1+C_{\epsilon}.
    \end{split}
\end{equation}
Taking $\widehat{\phi}=\frac{\phi}{ \parallel \phi \parallel_{H^{1} }}$, $c_{2}=\frac{1}{2(1+C_{\epsilon})}$ and $K_{2}=\overline{K}_{\epsilon}$ we obtain the result.
\end{proof}

We now turn to the proof of proposition \ref{localizingH-1norm}
\begin{proof}
(Of Proposition \ref{localizingH-1norm})
Again, since
\begin{equation}
    \parallel \nu \parallel_{H^{-1}}= \parallel -\nu \parallel_{H^{-1}},
\end{equation}
we assume without loss of generality that $\nu$ is positive outside of $K.$ Then by Claim \ref{extensionlemma} there exists a $\varphi \in H^{1}(K_{\epsilon})$ such that 
\begin{itemize}
    \item[$\bullet$] The function $\varphi$ has norm $1,$ i.e.
    \begin{equation}
        \parallel \varphi \parallel_{H^{1} }=1.
    \end{equation}
    \item[$\bullet$] The trace of $\varphi$ is positive, i.e.
    \begin{equation}
        \text{tr}(\varphi) \geq 0.
    \end{equation}
    \item[$\bullet$] We have that
    \begin{equation}
        \int_{K_{\epsilon}} \nu \varphi \geq  c \parallel \nu|_{K_{\epsilon}} \parallel_{H^{-1}(K_{\epsilon})}, 
    \end{equation}
    where $c$ depends only on $K$ and $\epsilon.$
\end{itemize}

By Claim \ref{extensionlemma}, there exists an extension $\widehat{\varphi}$ of $\varphi$ such that 
\begin{itemize}
    \item[$\bullet$] The support of $\widehat{\varphi}$ is contained in $\overline{K}_{1+\epsilon}$
    \item[$\bullet$] The norm of $\nabla \widehat{\varphi}$ is controlled by 
    \begin{equation}
        \parallel \nabla \widehat{\varphi} \parallel_{L^{2} } \leq k,
    \end{equation}
    where $k$ depends only on $K$ and $\epsilon.$
    \item[$\bullet$] We have that $\widehat{\varphi}$ is nonnegative in $K_{1+\epsilon} \setminus K_{\epsilon}.$
\end{itemize}
Since $\nu$ is positive outside of $K$ and $\widehat{\varphi}$ is positive outside of $K_{\epsilon},$ we have that 
\begin{equation}
    \int_{K_{\epsilon}} \nu \varphi \leq \int_{\mathbf{R}^{d}} \nu \widehat{\varphi}.
\end{equation}
Finally, we have that 
\begin{equation}
    \begin{split}
        \parallel \nu \parallel_{H^{-1}} &\geq \frac{1}{\parallel \nabla \widehat{\varphi} \parallel_{L^{2} } } \int_{\mathbf{R}^{d}} \nu \widehat{\varphi}\\
        &\geq \frac{1}{\parallel \nabla \widehat{\varphi} \parallel_{L^{2} } } \int_{K_{\epsilon}} \nu \varphi\\
        &\geq \frac{1}{k}  \parallel \nu|_{K_{\epsilon}} \parallel_{H^{-1}(K_{\epsilon})},
    \end{split}
\end{equation}
where $k$ depends only on $K$ and $\epsilon.$
\end{proof}

\section{Acknowledgements}

The author wants to thank Sylvia Serfaty, for many useful discussions.

The author has been funded by the Deutsche Forschungsgemeinschaft (DFG) - project number 417223351, and a McCracken Scholarship.

\bibliographystyle{plain}
\bibliography{bibliography.bib}

\end{document}